\documentclass[a4paper,10pt,reqno]{amsart}

\usepackage{amssymb,amsmath,amsthm,array,bm}
\usepackage{graphicx}
\usepackage{xcolor}
\usepackage{booktabs}
\usepackage{url}
\usepackage{arydshln}
\usepackage{multirow}
\usepackage{subcaption}
\usepackage[section]{placeins} 
\usepackage{comment}

\usepackage{mathtools}
\mathtoolsset{showonlyrefs=true}

\newtheorem{thm}{Theorem}[section]

\newtheorem{lem}[thm]{Lemma}

\newtheorem{rem}[thm]{Remark}
\newtheorem{ass}[thm]{Assumption}

\numberwithin{equation}{section}
\allowdisplaybreaks

\newcommand{\mcc}{\mathcal{C}}

\newcommand{\mcf}{\mathcal{F}}
\newcommand{\mcg}{\mathcal{G}}
\newcommand{\mci}{\mathcal{I}}
\newcommand{\mcj}{\mathcal{J}}
\newcommand{\mcl}{\mathcal{L}}

\newcommand{\mcn}{\mathcal{N}}

\newcommand{\mfb}{\mathfrak{b}}

\newcommand{\mfm}{\mathfrak{m}}

\newcommand{\mbbh}{\mathbb{H}}
\newcommand{\mbbn}{\mathbb{N}}

\newcommand{\mbbr}{\mathbb{R}}

\newcommand{\mbby}{\mathbb{Y}}

\newcommand{\al}{\alpha} \newcommand{\lam}{\lambda} \newcommand{\ep}{\epsilon} 
\newcommand{\vp}{\varphi} \newcommand{\del}{\delta} \newcommand{\sig}{\sigma}
\newcommand{\D}{\Delta} \newcommand{\Sig}{\Sigma} 
 \newcommand{\Gam}{\Gamma}
\newcommand{\vt}{\vartheta} 

\newcommand{\p}{\partial}
 
\newcommand{\cil}{\xrightarrow{\mcl}} 
\newcommand{\cip}{\xrightarrow{p}} 

\newcommand{\argmax}{\mathop{\rm argmax}}

\def\ds#1{\displaystyle{#1}}

\def\nn{\nonumber}

\def\tcr#1{\textcolor{black}{#1}}

\def\hmrev#1{\textcolor{black}{#1}}

\newcommand{\pr}{P} \newcommand{\E}{E}

\newcommand{\cov}{{\rm Cov}}
\newcommand{\diag}{\mathrm{diag}}

\def\wp{Wiener process}

\newcommand{\wh}{\widehat}

\newcommand{\taues}{\widehat{\tau}}

\newcommand{\sumj}{\sum_{j=1}^{n}}
\newcommand{\sumi}{\sum_{i=1}^{N}}

\newcommand{\ny}{\overline{y}}

\title[Quasi-likelihood inference for SDE with mixed-effects]{Quasi-likelihood inference for SDE with mixed-effects observed at high frequency}

\author[M. Delattre]{Maud Delattre}
\address{Universit\'{e} Paris-Saclay, INRAE, MaIAGE, 78350, Jouy-en-Josas, France.}

\author[H. Masuda]{Hiroki Masuda}
\address{Graduate School of Mathematical Sciences, University of Tokyo, 3-8-1 Komaba Meguro-ku Tokyo 153-8914, Japan.}
\email{hmasuda@ms.u-tokyo.ac.jp}

\date{\today}
\keywords{Non-Gaussian quasi-likelihood inference, normal variance-mean mixture, SDE with mixed effects.}

\begin{document}

\begin{abstract}
We consider statistical inference for a class of dynamic mixed-effect models described by stochastic differential equations whose drift and diffusion coefficients simultaneously depend on fixed- and random-effect parameters. Assuming that each process is observed at high frequency and the number of individuals goes to infinity, we propose a stepwise inference procedure and prove its theoretical properties. The methodology is based on suitable quasi-likelihood functions by profiling the random effect in the diffusion coefficient at the first stage, and then \hmrev{integrating out the Gaussian random effect in the drift coefficient to obtain the marginal distribution} in the second stage, resulting in a fully explicit and computationally convenient method. It is also the strength of the proposed approach that \hmrev{the proposed method allows a wide variety of distributions for the random effects in the diffusion coefficient}.
\end{abstract}

\maketitle


\section{Introduction}

Mixed-effects models play a pivotal role in various scientific domains, facilitating the precise analysis of repeated observations among individuals. Despite their utility, parameter estimation in these models presents significant technical challenges. However, their application improves our understanding of dynamic biological phenomena, making them invaluable tools in numerous fields. Moreover, there is a growing interest in employing mixed-effects models coupled with stochastic differential equations (SDE). Originating from efforts in pharmacokinetics (see \textit{, e.g.} references in \cite{donnet2013review}), where the aim was to refine the modeling of intrinsic variability in drug distribution within organisms, these models have now branched out into diverse applications within life sciences. From the intricate workings of neurobiology (\cite{ditlevsen2013introduction}) to the intricate dance of epidemic dynamics prediction (\cite{narci2022inference}), mixed-effects models offer invaluable insight. Tailored estimation methodologies have been developed to fully exploit the potential of these models. Although numerous algorithms are proposed in the literature (see \textit{e.g.} references in \cite{picchiniWeb}), addressing parameter estimation in these models from a theoretical point of view remains rare. Indeed, maximum likelihood estimation, a favored method, poses challenges due to its \hmrev{rarely} explicit form. Only in select works do the authors provide specific results. Ditlevsen and De Gaetano are mainly limited to Brownian motion with random drift \cite{ditlevsen2005mixed}. Then Delattre, Samson, Genon-Catalot, and Lar\'{e}do focused on real diffusions with specific characteristics and Gaussian and inverse gamma random effects in the drift and in the diffusion coefficient, respectively (\cite{DelGenLar18-2}, \cite{delattre2018parametric}, \cite{delattre2013maximum}, \cite{delattre2015estimation}). Covariates and Gaussian random effects in the drift were considered in \cite{GroSamDit20} in multivariate diffusion models. However, there still remains a gap in understanding the asymptotic behavior of parameter estimators in more general mixed-effects diffusion models.

Furthermore, prevalent assumptions in mixed-effects model literature, such as random effects following Gaussian distributions, may not always hold in practice. Recent advancements propose extensions, introducing distributions like the generalized hyperbolic distribution to better accommodate variability (see \textit{e.g.} \cite{AsaBolDigWal20} and \cite{FujMas23}). 

Our work revisits the notion of utilizing normal variance-mean mixture-type random effects in broader mixed-effects models defined by stochastic differential equations. This allows for the exploration of a wider class of mixed-effects diffusion models compared to the previous literature. We also consider that both drift and diffusion coefficients depend on random and fixed effects simultaneously. 
\hmrev{Moreover, an important contribution that sets us apart from the literature is that the diffusion coefficient may depend on both random and fixed effects}. 
We propose a novel parameter estimation method and provide theoretical insights into the asymptotic behavior of these estimators. Our estimation method diverges from the quasi-likelihood approach, offering a more accessible numerical procedure. This trade-off ensures stability and ease of implementation in high-frequency frameworks.
\hmrev{
Apart from computational efficiency, an advantage of the proposed estimation method is that we have a wide variety of distributions of the random-effect parameters $\tau_i$. Although several existing papers allow a general distribution of some or all random effects (e.g., \cite{PicDit11}, \cite{WGBS17}, and \cite{WGMP21}), the estimation methods therein rely on heavy numerical computations and do not establish theoretical properties, such as consistency and asymptotic normality.
}

The rest of this paper is organized as follows.
Section \ref{hm:sec_setup.assump} describes the precise model setup and assumptions.
Section \ref{hm:sec_p.m_qmlf} introduces the stepwise estimation strategy with two kinds of quasi-(log-)likelihood functions, and presents asymptotic properties of the associated estimators. In Section \ref{sec:simus}, we present numerical experiments illustrating the theoretical results, followed in Section \ref{sec:appli} by some empirical applications.
Section \ref{sec:discussion} provides some discussion.
Finally, Section \ref{hm:sec_proofs} presents the proofs of the main result.


\section{Setup and assumptions}
\label{hm:sec_setup.assump}

\subsection{Underlying dynamics}

Throughout, $T>0$ denotes a fixed constant.
Given an underlying filtered probability space $(\Omega,\mcf,(\mcf_t)_{t\le T},\pr)$ satisfying the usual hypothesis, we consider the following system of \hmrev{one-dimensional} stochastic integral equations:
\begin{align}
Y_i(t) &= Y_i(0) + \tau_i \,
\left(\vp_{f} \cdot \int_0^t a_f(Y_i(s))ds + \vp_{r,i} \cdot \int_0^t a_r(Y_i(s))ds \right)
\nn\\
&{}\qquad +\sqrt{\tau_i}\int_0^t c(Y_i(s);\eta) dw_i(s),
\label{eq:Y_sde}
\end{align}
for $i=1,2,\dots$ and $t\in[0,T]$, where the ingredients are given as follows.

\begin{itemize}

\item The driving noise processes 
$w_1(\cdot),w_2(\cdot),\dots$ are i.i.d. standard one-dimensional {\wp es} independent of the i.i.d. initial random variables $Y_1(0),Y_2(0),\dots$ in $\mbbr$. We assume that for every $K>0$,
\begin{equation}\label{hm:X0Y0-moments}
    \E\big[|Y_1(0)|^K\big] <\infty.
\end{equation}

\item $\eta \in\Theta_\eta\subset\mbbr^{p_\eta}$ and $\vp_f \in \Theta_{\vp_f}\subset\mbbr^{p_{\vp_f}}$ are fixed-effect parameters.

\item $(\tau_1,\vp_{r,1}), (\tau_2,\vp_{r,2}), \dots\in (0,\infty) \times \mbbr^{p_{\vp_r}}$ are i.i.d. random-effect parameters whose common distribution is assumed to satisfy the following conditions.
\begin{itemize}
    \item $\tau_i$ and $\vp_{r,i}$ are independent.
    \item $\mcl(\tau_i)$ admits the parametrized Lebesgue density
    \begin{equation}
        \tau_i \sim f(\tau;\theta_\tau)d\tau
    \end{equation}
     on $(0,\infty)$, where the parameter $\theta_\tau\in\Theta_\tau\subset\mbbr^{p_\tau}$;
    \hmrev{here and in what follows, the symbol $\mcl(\xi)$ for a random variable $\xi$ denotes the distribution of $\xi$.}

    \item $\mcl(\vp_{r,i})$ is non-degenerate Gaussian:
    \begin{equation}
        \vp_{r,i} \sim N_{p_{\vp_r}}(0,\Sig_r)
        \label{hm:vp.iid}
    \end{equation}
    for a positive definite $\Sig_r\in\mbbr^{p_{\vp_r}} \otimes \mbbr^{p_{\vp_r}}$.
\end{itemize}

\item 
\hmrev{$a_f(y):\,\mbbr\to \mbbr^{p_{\vp_f}}$, $a_r(y):\,\mbbr\to \mbbr^{p_{\vp_r}}$, and $c(y;\eta):\,\mbbr\times\Theta_\eta\to\mbbr$ are known measurable functions except for the parameter $\eta=(\eta_k)_{k=1}^{p_\eta}$; in addition, we will write $a(y)=(a_f(y),a_r(y)):\,\mbbr\to \mbbr^{p_{\vp_f}+p_{\vp_r}}$.}

\end{itemize}

\hmrev{Although our model \eqref{eq:Y_sde} may be less general than some other SDE with mixed effects, it provides an interpretable and computationally efficient estimation framework that takes advantage of the inherent mathematical structure of the model.} 
The random-effect parameter $\tau_i$ represents the model-time scale of the $i$th individual relative to the common period $[0,T]$. 
\hmrev{The estimation method presented below remains equally valid even if the drift depends on the random effect parameter $\tau_i$ in more complex ways}: see Remarks \ref{hm:rem_diff-1} and \ref{hm:rem_re-2}.
Given the individual model-time scales, the random-effect parameter $\vp_{r,i}$ represents variability in trend characteristics across $N$ individuals.
\hmrev{Due to the multiplicative form $\tau_i \vp_{r,i}$, the random-effect parameter in the drift coefficient is a normal variance-mean mixture, hence in particular can flexibly exhibit heavy-tailed character.}

Let $p_\vp:=p_{\vp_f}+p_{\vp_r}$.
To compactify the notation, we denote
\begin{equation}\label{hm:vp.param_notation}
\begin{split}
    & \vp_i=(\vp_f,\vp_{r,i})\in\mbbr^{p_\vp}, \\
    & \mu=(\vp_{f},0)\in\mbbr^{p_\vp}, \quad \Sig=\diag(O,\Sig_r)\in\mbbr^{p_\vp}\otimes\mbbr^{p_\vp}.
\end{split}
\end{equation}
That is, we assume that we know a priori which component of the random effect parameter $\vp_i$ is truly random; if not, the statistical model exhibits a non-standard feature that the true value of the covariance matrix $\Sig$ can lie on the boundary of the parameter space and we will not touch this model-selection problem here; in such case, the regularization will help, see, for example, \cite{BonKriGho10}, \cite{FanLi12}, and \cite{YosYos24}.
We note, however, that our setting is implicitly able to consider the cases of non-degenerate $\vp_i$ when $a_f(y)=a_r(y)$ (hence $p_{\vp_f}=p_{\vp_r}$); in this case, we may conveniently write $\vp_i \sim N_{p_{\vp_f}}(\vp_f,\Sig_r)$.
\hmrev{We also remark that requiring $\vp_{r,i}$ to have a zero mean in \eqref{hm:vp.iid} is not a substantial restriction: if $\vp_{r,i} \sim N_{p_{\vp_r}}(\mu_r,\Sig_r)$ for a possibly non-zero $\mu_r$, then the identity $\vp_f \cdot a_f(y) + \vp_{r,i}\cdot a_r(y) = (\vp_f,\mu_r) \cdot (a_f(y), a_r(y)) + (\vp_{r,i} - \mu_r) \cdot a_r(y)$ reduces the problem to the original zero-mean setting.}

All the stochastic processes introduced above will be $(\mcf_t)$-adapted; in \eqref{hm:Ft_def} below, we will impose a specific structural assumption on $\mcf_t$. 
We denote by $\vt$ the statistical parameter of interest:
\begin{align}
    \vt:=(\eta,\theta_\tau,\vp_f,\Sig_r) 
    &\in \Theta_\eta \times \Theta_\tau \times \Theta_{\vp_f} \times \Theta_{\Sig_r} 
    \nn\\
    &\subset \mbbr^{p_\eta}\times\mbbr^{p_\tau}\times\mbbr^{p_{\vp_f}}\times\mbbr^{p_{\vp_r}\times p_{\vp_r}}.
\end{align}
We write $\Theta:=\Theta_\eta \times \Theta_\tau \times \Theta_{\vp_f} \times \Theta_{\Sig_r}$ for the full-parameter space. We assume that $\Theta_\eta \times \Theta_\tau$, $\Theta_{\vp_f}$, and $\Theta_{\Sig_r}$ are a bounded convex domain and that any $A\in\overline{\Theta_{\Sig_r}}$ is positive definite. 
We denote the true value of $\vt$ by 
\begin{equation}
    \vt_0:=(\eta_{0},\theta_{\tau,0},\vp_{f,0},\Sig_{r,0}) \in \Theta,
\end{equation}
assumed to exist; also, we will write $\mu_0=(\vp_{f,0},0)$ and $\Sig_0=\diag(O,\Sig_{r,0})$.

With the setup mentioned above, our objective is to estimate $\vt_0$ based on a high-frequency sample from $N$ independent individuals:
\begin{equation}
\left\{Y_i(t_{j}):~ 1\le i\le N,\quad 1\le j\le n\right\},
\nonumber
\end{equation}
where $t_{j}=jh$ with $h=h_n:=T/n$ for a fixed $T>0$; that is, all individuals are observed at equally spaced time points.

\subsection{Basic notation}

In the rest of this paper, we will use the following notation and conventions.

\begin{itemize}
\item $C$ denotes a universal positive constant, which may change at each appearance.
\item $a_{N,n} \lesssim b_{N,n}$ means \hmrev{$a_{N,n}\le Cb_{N,n}$} for every $n,N$ large enough.
\item $\p_{a}^{k}$ denotes the $k$-times partial differentiation with respect to variable $a$.
\item For a matrix $A$, we write $A^{\otimes 2} = AA^\top$, where $\top$ stands for the transposition, and $\lam_{\min}(A)$ denotes the minimum eigenvalue of $A$ if it is square.

\item $\phi_k(\cdot;m,V)$ for positive definite $V$ denotes the Gaussian $N_k(m,V)$-density; for $k=1$, we omit the subscript $k=1$ with $\phi(\cdot):=\phi(\cdot;0,1)$.

\item $S(y;\eta):=c^2(y;\eta)$, the squared diffusion coefficient.

\item For any process $\xi_i(t)$,
\begin{equation}
    \xi_{ij}:=\xi_i(t_{j}), \quad \xi_{i,j-1}:=\xi_i(t_{j-1}), \quad 
    \D_j \xi_i := \xi_{ij}-\xi_{i,j-1}.
\end{equation}

\item For any measurable function $\mathsf{f}$ on $\mbbr\times\overline{\Theta}$, we write
\begin{align}
 \mathsf{f}_{i}(t;\vt) &= \mathsf{f}(Y_i(t);\vt),
 \nn\\
 \mathsf{f}_{i,j-1}(\vt) &= \mathsf{f}(Y_i(t_{j-1});\vt),
\end{align}
with simply $\mathsf{f}_{i}(t)$ and $\mathsf{f}_{i,j-1}$, respectively, when $\vt=\vt_0$ or the function $\mathsf{f}$ does not depend on $\vt$; 
for example, $S_i(t;\eta):=S(Y_i(t);\eta)$, $S_i(t):=S_i(t;\eta_0)$, $S_{i,j-1}:=S_i(t_{j-1})$, and so on.

\item The normalized $j$th increment of $Y_i$ is denoted by
\begin{equation}
    \ny_{ij} = h^{-1/2}\D_j Y_i.
\end{equation}

\item $\int_{j}:=\int_{t_{j-1}}^{t_{j}}=\int_{(j-1)T/n}^{jT/n}$.
\end{itemize}

\subsection{Assumptions}

This section summarises the regularity conditions assumed throughout this paper. Unless otherwise mentioned, all the asymptotics will be taken for $n\to\infty$ and $N=N_n\to\infty$ in such a way that
\begin{equation}\label{hm:sampling.design}
    n^{\mathsf{a}'} \lesssim N \lesssim n^{\mathsf{a}''},
\end{equation}
where $\mathsf{a}'$ and $\mathsf{a}''$ are positive constants satisfying that $0<\mathsf{a}'\le\mathsf{a}''<1$.
In particular, \eqref{hm:sampling.design} implies that 
\begin{equation}\label{hm:N/n->0}
    \frac{N}{n}\to 0.
\end{equation}

\begin{ass}[Regularity of the coefficients]
\label{hm:A_regul}
We have
    \begin{equation}\nn
        \max_{k\in\{0,1,2\}} \max_{l\in\{0,1,2,3\}}
        \hmrev{\sup_y}\left(
        \big|\p_{y}^k a(y)\big| + \sup_{\eta}\big|\p_{\eta}^l \p_{y}^k c(y;\eta)\big| 
        \right) < \infty,
    \end{equation} 
    \begin{equation}
        \inf_{(y,\eta)}c(y;\eta)>0.
    \end{equation}
\end{ass}

We assume that
\begin{equation}\label{hm:Ft_def}
    \mcf_t = \bigvee_{i\in\mbbn}\mcf_{i,t},
    \qquad
    \mcf_{i,t}:=\sig\left(\tau_i,\vp_{r,i},Y_i(0),w_i(s):\,s\le t\right) \vee \mcn
\end{equation}
for $t\le T$ for the underlying filtration $(\mcf_t)_{t\le T}$, where $\mcn$ denotes the set of all $\mcf$-measurable $\pr$-null sets.

Note that the random elements $(\tau_1,\vp_1,Y_1(\cdot)),(\tau_2,\vp_2,Y_2(\cdot)),\dots$ are i.i.d. 
As in \cite[Lemma 1]{delattre2018parametric}, by It\^{o}'s formula with the moment conditions \eqref{hm:X0Y0-moments} and \eqref{hm:tau-moments} below, and also by the boundedness of $(a,b,c)$ and the Burkholder inequality, it is not difficult to verify that for every $K>0$,
\begin{align}
& \sup_{t\in[0,T]} \E\big[|Y_1(t)|^K\big] 
\nn\\
&{}\qquad 
+ \sup_{n\ge 1}\max_{j\le n}
\sup_{t\in(t_{j-1},t_{j}]}
\E\left[ \left| \frac{1}{\sqrt{h_n}}(Y_1(t) - Y_1(t_{j-1})) \right|^K \right]
< \infty.
\label{hm:y.ij_mb}
\end{align}

\begin{ass}[Identifiability]
\label{hm:A_iden}
The following holds for each $i\le N$:
\begin{enumerate}
        \item With probability one, the function $\ds{t \mapsto \frac{S_i(t;\eta_{0})}{S_i(t;\eta)}}$ is constant on $[0,T]$ if and only if $\eta=\eta_{0}$;
        \item With probability one, the functions $\ds{t \mapsto \frac{\p_{\eta_k}S_i(t;\eta_{0})}{S_i(t;\eta_{0})}}$ ($k=1,\dots,p_\eta$) are not constant on $[0,T]$;
        \item The matrix $\ds{\int_0^T \frac{a_i^{\otimes 2}(t)}{S_i(t;\eta_{0})}dt}$ $(i\ge 1)$ is a.s. positive definite.
\end{enumerate}
\end{ass}

\begin{rem}
\label{hm:rem_X}
~
\begin{enumerate}
    \item We assumed the boundedness of the coefficients in Assumption \ref{hm:A_regul}. 
    \hmrev{This may be rather restrictive, for it excludes the Ornstein-Uhlenbeck process with mixed effects. 
    However, the unboundedness of the support prevents us from using the standard Gronwall-inequality argument to deduce the boundedness of moments $\E[|Y_i(t)|^K]$ and $\E[|\ny_{ij}|^K]$. This, in turn, makes it difficult to establish the estimate \eqref{hm:y.ij_mb}, which is essential in our technical proofs; for example, dropping the dependence on $\tau_i$ from the drift coefficient does not resolve this issue.}
    Still, as was mentioned in \cite[p.1932]{delattre2018parametric}, if the support of $\mcl(\tau_1,\vp_1)$ in $\mbbr^{1+p_\vp}$ is essentially bounded, then \hmrev{the first condition 
    in Assumption \ref{hm:A_regul} could be relaxed.}

    \item The second item in Assumption \ref{hm:A_iden}, which says that the function $x\mapsto\p_{\eta} \log c(y;\eta)$ is not identically zero, is required to identify $\eta$ in the presence of the random-effect parameter $\tau_i$; for example, for a known and non-constant function $(c_0,c_1)$, this includes $c(y;\eta)=\exp\{ c_0(y)+\eta \cdot c_1(y)\}$ \hmrev{for some bounded $c_0(y)$ and $c_1(y)$,} but excludes the multiplicative case $c(y;\eta)=\eta c_0(y)$. 
    \hmrev{The condition is required because we must estimate the fixed-effect parameter $\eta$ in the presence of the multiplicative random-effect parameter $\tau_i$. Specifically, we need $\eta$ to depend on $c(x,y;\eta)$ in a non-linear manner to avoid the term ``$\sqrt{\tau_i}\,\eta$'' which obviously causes an identifiability issue in estimating $\mcl(\tau_i)$ and $\eta$ separately. The second item in Assumption \ref{hm:A_iden} embodies this requirement.} 
    These arguments are in line with those in \cite{EguMas19}, while there is some additional complexity here due to the presence of the random-effect parameters.

    \item It is straightforward to extend the model \eqref{eq:Y_sde} for $Y_i$ to a multivariate one, say in $\mbbr^d$, where $a(y)$ is $\mbbr^d \otimes \mbbr^{p_{\vp}}$-valued and $c(y;\eta)$ is $\mbbr^d \otimes \mbbr^{r}$-valued and where $w_i$ is an $r$-dimensional standard {\wp}.
\end{enumerate}
\end{rem}

We need some assumptions on the distribution $\mcl(\tau_1)$.

\begin{ass}[Diffusion random-effect distribution]
\label{hm:A_rand.eff}
There exists a constant $C_\tau>0$ for which
        \begin{align}
            \max_{k=0,1,2,3}\max_{l=0,1}\sup_{\theta_\tau}\left| \p_{\theta_\tau}^k \p_\tau^l \log f(\tau;\theta_\tau)\right| \lesssim 1+\tau^{C_\tau}+\tau^{-C_\tau},
        \end{align}
        and $f(\tau;\theta_\tau)=f(\tau;\theta_{\tau,0})$ ($\mcl(\tau_1)$-a.e.) if and only if $\theta_\tau=\theta_{\tau,0}$. Moreover, 
        \begin{equation}\label{hm:tau-moments}
        \forall K_\tau\in\mbbr,\quad 
        \sup_{\theta_\tau\in\overline{\Theta_\tau}}
        \int_{0+}^\infty t^{K_\tau}f(t;\theta_\tau)dt
        < \infty.
        \end{equation}
\end{ass}

\section{Profile-marginal quasi-likelihood inference}
\label{hm:sec_p.m_qmlf}

In the estimation of a single diffusion-type model, it is possible to conduct a stepwise inference through profiling under a sufficiently high-frequency sampling scheme. It first focuses on the diffusion coefficient while ignoring the drift term, often enabling us to conduct an adaptive inference by reducing the computational costs in (numerical) optimizations:
among others, we refer to \cite{UchYos12} and \cite{KamUch15} for details.

In the mixed-effects framework, the previous study \cite[Section 3]{FujMas23} developed a simple three-stage conditional least-squares estimation procedure in a class of generalized hyperbolic type location-scale mixed-effects models. However, it will be impossible to directly apply the same strategy as in \cite{FujMas23} to our model \eqref{eq:Y_sde} since each conditional distribution of $Y_i$ given $(\vp_i,\tau_i)$ is dynamic, so that the filtration structure is much more complicated to handle.
We will further modify the stepwise estimation procedure by profiling the random-effect parameters to obtain computationally convenient estimators; 
the strategy is possible due to the specific form of the coefficients of our model setup.
Relatedly, \cite[Section 4]{GenLar16} studied the estimation of the distribution of the random-effect parameter in the drift coefficient based on continuous records. There, plug-in-type parametric estimation was considered for, as a specific example,  the positive random-effect parameter $\vp_i\sim\Gam(a,\lam)$ in $dY_i(t)=-\vp_i Y_i(t)dt + dW_i(t)$ based on continuous records $\{(Y_i(t))_{t\le T})\}_{i\le N}$ for large $T$ and $N$.
This work is an important theoretical study. However, in our discrete-time sampling setting, we essentially need different technical tools.

Specifically, we will proceed as follows.
\begin{itemize}
\item First, ignoring the drift term, we focus on 
\hmrev{estimating the parameters in the diffusion coefficient:}
\begin{equation}
\vt_1:=(\eta,\theta_\tau).
\nonumber
\end{equation}
\item Next, we turn to estimating the remaining 
\hmrev{parameters in the drift coefficient:}
\begin{equation}
\vt_2:=(\vp_f,\Sig_r)
\nonumber
\end{equation}
with plugging in the estimated value of $\vt_1$.
\end{itemize}
The distribution of the random-element sequence $\{(Y_i,w_i,\tau_i,\vp_i)\}_{i\ge 1}$ induces the family of probability measures $(\pr_\vt)_{\vt\in\Theta}$.
We will write $\pr$ for $\pr_{\vt_0}$; this is an abuse of notation, but it will not cause any inconvenience.

\subsection{Diffusion parameter: profile quasi-likelihood}
\label{hm:sec-diff.param}

In this first step, we regard the drift components in \eqref{eq:Y_sde} as nuisance elements in estimating the target parameter $\vt_1$.
Recall that we are using the notation:
\begin{equation}
    \vp_i\cdot a(y)=\vp_f \cdot a_f(y) + \vp_{r,i}\cdot a_r(y).
\end{equation}
By applying the Euler approximation and then ignoring the drift parts, the approximate $\tau$-conditionally Gaussian discrete-time dynamics is formally given by ($\overset{\vt}{=}$ indicates the identity under $\pr_\vt$)
\begin{align}
Y_{ij} &\overset{\vt}{=} Y_{i,j-1}+ \tau_i\,\vp_i \cdot \int_j a(Y_i(s))ds 
 +\sqrt{\tau_i}\int_j c(Y_i(s);\eta) dw_i(s)
\nn\\
&\approx Y_{i,j-1}+ \tau_i\,\vp_i \cdot a_{i,j-1}h +\sqrt{\tau_i}\, c_{i,j-1}(\eta) \D_j w_i
\nn\\
& 
\approx Y_{i,j-1} + \sqrt{\tau_i} \,c_{i,j-1}(\eta) \D_j w_i.
\label{hm:step.diff-1}
\end{align}
We will construct our estimator of $\vt_1$ as follows.
\hmrev{
\begin{itemize}
    \item First, by profiling $\tau_i$ out of the random function \eqref{hm:rev_add1} below, we obtain a new random function \eqref{hm:def_QLF1-1} that depends only on $\eta$; our estimator of $\eta$ is then defined as its maximizer.
    \item Next, after constructing the predictors of $\tau_i$ by plugging the resulting estimator $\wh{\eta}$ of $\eta$ into $\wh{\tau}_i(\eta)$ given in \eqref{hm:tau-hat} below, we optimize the estimated marginal likelihood \eqref{hm:def_QLF1-2} below with respect to $\tau_1,\dots,\tau_N$ to estimate $\theta_\tau$.
\end{itemize}
}

Recall the notation $S_{i,j-1}(\eta)=c_{i,j-1}^2(\eta)$.
If we temporarily regard the random-effect parameters $\{\tau_i\}_{i}$ as statistical (non-random) parameters, then it is natural from \eqref{hm:step.diff-1} to consider the following Gaussian quasi-likelihood:
\begin{equation}\label{hm:rev_add1}
(\eta,\tau_1,\dots,\tau_N) \mapsto
\sum_{i,j} \log\phi\big(\ny_{ij}; \hmrev{0}, h\tau_i \,S_{i,j-1}(\eta)\big).
\end{equation}
Solving the associated quasi-likelihood equations with respect to $\tau_1,\dots,\tau_N$ leads to the following explicit solutions as random function of $\eta$:
\begin{equation}\label{hm:tau-hat}
\taues_i(\eta) :=\frac1n \sum_j S^{-1}_{i,j-1}(\eta) \,\ny_{ij}^2,\qquad i=1,\dots,N.
\end{equation}

If $\{\tau_i\}_{i}$ were observed in addition to $\{\ny_{ij}\}_{i,j}$, then taking the $\tau_i$-conditional approximation \eqref{hm:step.diff-1} into account we may define the quasi-log-likelihood as follows:
\begin{equation}
(\eta,\theta_\tau) \mapsto 
\sum_{i,j} \log\phi\left(Y_{ij}; Y_{i,j-1}, h\tau_i \,S_{i,j-1}(\eta)\right) + \sum_i \log f(\tau_i;\theta_\tau).
\label{hm:clf-1}
\end{equation}
However, since $\{\tau_i\}_{i}$ is unobserved, we substitute the expressions \eqref{hm:tau-hat} into $\tau_i$ in \eqref{hm:clf-1} to obtain the random function of the non-random quantity $\vt_1=(\eta,\theta_\tau)$:
\begin{align}
(\eta,\theta_\tau) \mapsto 
\sum_{i,j} \log\phi\left(Y_{ij}; Y_{i,j-1}, h\taues_i(\eta) \,S_{i,j-1}(\eta)\right) + \sum_i \log f(\taues_i(\eta);\theta_\tau).
\label{hm:clf-1.2}
\end{align}
Ignoring the irrelevant terms free from $\vt_1$ in \eqref{hm:clf-1.2}, we define the following random function, which we call the \textit{first-stage profile quasi-likelihood}:
\begin{equation}
\mbbh_{1,N,n}(\vt_1) := \mbbh_{11,N,n}(\eta) + \mbbh_{12,N,n}(\vt_1),
\label{hm:def_QLF1}
\end{equation}
where
\begin{align}
\mbbh_{11,N,n}(\eta) &:=
 -\frac12 \sum_i \bigg\{ \sum_{j} \log S_{i,j-1}(\eta) + n\log\bigg(\frac1n \sum_j S_{i,j-1}^{-1}(\eta)\,\ny_{ij}^2\bigg)\bigg\},
\label{hm:def_QLF1-1}\\
\mbbh_{12,N,n}(\vt_1) &:= \sum_i \log f(\taues_i(\eta);\theta_\tau).
\label{hm:def_QLF1-2}
\end{align}
In what follows, we will largely omit the dependence on ``$N,n$'' from the notation, such as $\mbbh_{1}(\vt_1)=\mbbh_{1,N,n}(\vt_1)$.

One could define an estimator of $\vt_1$ by any maximizer of $\mbbh_{1}$.
However, it will be seen that the random function $\mbbh_{1}(\vt_1)$ has the mixed-rates structure: roughly speaking, $\mbbh_{11}(\eta)$ is of order $O_p(nN)$ while $\mbbh_{12}(\vt_1)$ is of order $O_p(N)$.
In the present asymptotics,
it will be shown that the fixed-effect parameter $\eta$ can be estimated more quickly than the parameter indexing $\mcl(\tau_i)$. It is therefore natural to estimate $\vt_1$ by $\wh{\vt}_1=(\wh{\eta},\wh{\theta}_\tau)$ defined by any
\begin{equation}
\wh{\eta} \in \argmax_{\eta} \mbbh_{11}(\eta),\qquad \wh{\theta}_\tau \in \argmax_{\theta_\tau} \mbbh_{12}(\wh{\eta},\theta_\tau).
\label{hm:def_stage-1.estimators}
\end{equation}
\hmrev{
Although the above construction may seem ad hoc, we will show in Theorem \ref{hm:thm_diff.param} below that it is theoretically justified, yielding asymptotically normal estimators.
}
The existence of $\wh{\vt}_1$ follows from the uniform continuity $(\mbbh_{11},\mbbh_{12})$ over $\overline{\Theta}_1$ and the measurable selection theorem.
As such, we can execute the optimization of $\theta_\tau\mapsto\mbbh_{12}(\wh{\eta},\theta_\tau)$ as if we observe the random effects $\tau_1,\dots,\tau_N$ as $\wh{\tau}_1(\wh{\eta}),\dots,\wh{\tau}_N(\wh{\eta})$, respectively. 
This estimation method will be computationally much more convenient than the direct optimization of $\vt_1 =(\eta,\theta_\tau)\mapsto \mbbh_{1}(\vt_1)$.

\begin{rem}
\label{hm:rem_diff-1}
~
\begin{enumerate}
    \item The random function $\mbbh_{11}(\eta)$ based on \eqref{hm:tau-hat} is similar to the \hmrev{Gaussian quasi-likelihood} for a discretely observed diffusion process with an unknown time scale \cite{EguMas19}, where we considered profiling out the time scale parameter from the conventional Gaussian quasi-likelihood and derived the asymptotic properties of the associated estimators. Indeed, our technical arguments in this paper partly owe to those in \cite{EguMas19}, while our situation is more complicated because of the presence of the random-effect parameters.    
    We also note that the linear dependence of $\tau_i$ in the drift coefficient of \eqref{eq:Y_sde} is never essential: indeed, it does not matter how $\tau_i$ depends on the drift coefficient in the second-stage drift estimation, where we will plug in $\wh{\tau}_i$ as a known quantity.
    See also Remark \ref{hm:rem_re-2}.
    
    \item In case where $c(x;\eta)$ does not depend on the unknown parameter $\eta$, then we can directly use
    \begin{align}\nn
    \mbbh_{1}(\theta_\tau) &:= \sum_i \log f(\taues^\ast_i;\theta_\tau)
    \end{align}
    to estimate $\theta_\tau$ based on the statistics $\taues^\ast_i :=n^{-1} \sum_j S^{-1}_{i,j-1} \,\ny_{ij}^2$.
\end{enumerate}
\end{rem}

Let
\begin{equation}\label{hm:def_g}
    g(y;\eta) := \p_{\eta}\log S(y;\eta) = \frac{\p_{\eta}S(y;\eta)}{S(y;\eta)},
\end{equation}
and then let ($g_1(t):=\hmrev{g}(Y_1(t);\eta_0)$)
\begin{align}
    Q_{11,0} &:= \frac12 \E\left[
    \frac1T \int_0^T g_1(t)^{\otimes 2}dt 
    - \left(\frac1T \int_0^T g_1(t)dt\right)^{\otimes 2}
    \right].
\end{align}
In view of the Cauchy-Schwarz inequality, this $Q_{11,0}$ is positive-definite under Assumption \ref{hm:A_iden}. 
Further, let
\begin{equation}
    \mci_{12} = \mci_{12}(\theta_{\tau,0}) := 
    \E\left[\left\{\p_{\theta_\tau}\log f(\tau_1;\theta_{\tau,0})\right\}^{\otimes 2}\right]
    \label{hm:def_I.12}
\end{equation}
denote the Fisher information matrix of the model $\{f(\cdot;\theta_\tau):\,\theta_\tau\in\Theta_\tau\}$.
The following theorem shows the asymptotic distribution of $\wh{\vt}_1$.

\begin{thm}\label{hm:thm_diff.param}
Under the setup described in Section \ref{hm:sec_setup.assump}, we have
\begin{equation}
\left(
\sqrt{nN}(\wh{\eta} -\eta_{0}),\, \sqrt{N}(\wh{\theta}_\tau -\theta_{\tau,0})
\right)
\cil N_{p_{\eta}+p_{\tau}}\left(0,\,\diag\left(Q_{11,0}^{-1},\, \mci_{12}^{-1}\right)\right).
\nonumber
\end{equation}
\hmrev{
Moreover, we can construct a consistent estimator of the asymptotic covariance matrix based on
\begin{align}
    \wh{Q}_{11} &:= \frac{1}{2N}\sumi \bigg\{
    \frac1n \sum_j g_{i,j-1}(\wh{\eta})^{\otimes 2} - \bigg(\frac1n \sum_j g_{i,j-1}(\wh{\eta})\bigg)^{\otimes 2}
    \bigg\} \cip Q_{11,0},
    \label{hm:hat-Q11}\\
    \wh{\mci}_{12} &:= 
    \frac{1}{N} \sumi \left(\p_{\theta_\tau}\log f(\wh{\tau}_i;\wh{\theta}_{\tau})\right)^{\otimes 2}
    \cip \mci_{12}.
    \label{hm:hat-I12}
\end{align}}
\end{thm}

The proof is given in Section \ref{hm:sec_proof.diff}. 

\begin{rem}
We here focused on the fully parametric log-likelihood function to make an inference for $\mcl(\tau_1)$. Obviously, instead of \eqref{hm:def_QLF1-2} we could consider any other $M$-estimation framework, say
\begin{equation}
\mbbh'_{12,N,n}(\vt_1) := \sum_i \rho(\taues_i(\eta);\theta_\tau),
\end{equation}
where $\rho:\,(0,\infty)\times \Theta_{\tau} \to \mbbr$ is a suitable objective function. 
\hmrev{
Then, the parameter $\theta_\tau$ does not necessarily determine the distribution of $\tau$, as is typical in semiparametric inference, the (generalized) method of moments, and related frameworks.
}
\end{rem}

\begin{rem}
If we are only interested in estimating the diffusion coefficients, we could deal with
\begin{align}
Y_i(t) &= Y_i(0) + \int_0^t \mu_i(s)ds +\sqrt{\tau_i}\int_0^t c(Y_i(s);\eta) dw_i(s),
\nn    
\end{align}
where $\mu_i$ are unknown $(\mcf_{i,t})$-adapted processes. This corresponds to a sort of population modeling of the non-ergodic volatility regression \cite{UchYos13}. 
The claims in Theorem \ref{hm:thm_diff.param} remain valid under mild conditions on $\mu_i$.
\end{rem}

\subsection{Drift parameter: marginal quasi-likelihood}
\label{hm:sec-drif.param}

Having obtained the first-stage estimate $\widehat{\vt}_1=(\widehat{\eta},\widehat{\theta}_\tau)$, we proceed to estimating the remaining drift parameter $\vt_2=(\vp_f,\Sig_r)$. 
The Euler approximation in this second stage takes the drift term into account:
\begin{equation}
Y_{ij} \overset{\vt}{\approx} Y_{i,j-1} + \tau_i\,\vp_i \cdot a_{i,j-1} h + \sqrt{\tau_i} \,c_{i,j-1}(\eta) \D_j w_i.
\nn
\end{equation}
We are assuming that $\vp_i=(\vp_f,\vp_{r,i}) \sim \text{i.i.d.}~N_{p_{\vp}}(\mu,\Sig)$, which is truly degenerate as soon as we have the fixed-effect parameter $\vp_f$ while $\vp_{r,i}$ is assumed to be non-degenerate: recall \eqref{hm:vp.iid} and \eqref{hm:vp.param_notation}.

Let us abbreviate as $\widehat{S}_{i,j-1}=c^2_{i,j-1}(\wh{\eta})$ and $\wh{\tau}_i = \wh{\tau}_i(\wh{\eta})$. 
\hmrev{Suppose temporarily} that $\{\tau_i\}$ and $\eta_0$ are known quantities given by $\{\wh{\tau}_i\}$ and $\wh{\eta}$, respectively, \hmrev{and also regard $\vp_i$'s as (non-random) statistical parameters to be estimated. Then, we could} consider the following auxiliary Gaussian quasi-likelihood of $\{\vp_i\}$:
\begin{align}
    \{\vp_i\} &\mapsto 
    \prod_i \exp\left( \sum_{j} \log \phi\left(Y_{ij}; Y_{i,j-1} + h \wh{\tau}_i\,\vp_i \cdot a_{i,j-1},\, h \wh{\tau}_i \,\wh{S}_{i,j-1}\right)
    \right)
    \nn\\
    &=: \prod_i \Bigg\{
    \Bigg(\prod_j \big(2\pi h \wh{\tau}_i \,\wh{S}_{i,j-1}\big)^{-1/2}\Bigg)
    \nn\\
    &{}\qquad \times 
    \exp\Bigg(- \frac12 \wh{\tau}_i^{-1} \frac1h \sum_{j} 
    \wh{S}_{i,j-1}^{-1}\left(\D_j Y_i - h \wh{\tau}_i a_{i,j-1}\cdot \vp_i\right)^2
    \Bigg) \Bigg\}
    \nn\\
    &=: \prod_i \wh{g}_i(\vp_i),\quad \text{say.}
\end{align}
\hmrev{However, since each $\vp_i$ is random and is not directly observed, we consider integrating them out with respect to the assumed Gaussian distribution.} 
As in \cite{delattre2018parametric} and \cite{DelGenLar18-2}, 
\hmrev{although $\Sig=\diag(O,\Sig_r)$ is always degenerate in our situation, we formally integrate out the variables $\vp_1,\dots,\vp_N$ with respect to the non-degenerate Gaussian distribution $N_{p_{\vp}}(\mu,\Sig)^{\otimes N}$}: 
for each $i\le N$, some matrix manipulations and the integration of the Gaussian density give
\begin{align}
    \vt_2 &\mapsto 
    \int_{\mbbr^{p_{\vp}}} \wh{g}_i(\vp_i) \phi_{p_{\vp}}(\vp_i;\mu,\Sig) d\vp_i
    \nn\\
    &= \wh{C}_{i} \det\big(\wh{M}_i^{-1} + \Sig\big)^{-1/2}
    \nn\\
    &{}\qquad \times 
    \exp\left( -\frac12 (\wh{M}_i^{-1} \wh{v}_i - \mu)^\top (\wh{M}_i^{-1}+\Sig)^{-1}(\wh{M}_i^{-1} \wh{v}_i - \mu) \right),
    \label{hm:exp.qlm-2}
\end{align}
where $\wh{C}_i$ denotes an $\mcf_{i,T} \vee \sig(\wh{\vt}_1)$-measurable random variable whose expression does not depend on $(\mu,\Sig,\vp_i)$, and where
\begin{align}
    \wh{M}_i &:= \wh{\tau}_i\,h \sum_j \wh{S}_{i,j-1}^{-1} a_{i,j-1}^{\otimes 2},
    \nn\\
    \wh{v}_i &:= \sum_j \wh{S}_{i,j-1}^{-1} a_{i,j-1} \D_j Y_i
\end{align}
are statistics. Note that $\wh{M}_i$ is a.s. non-negative definite. 
\hmrev{We also emphasize that the quantities $(\widehat{M}_i,\widehat{v}_i)$ depend on $\widehat{\eta}$, implying that this explicit estimation procedure is a natural consequence of the stepwise nature of the proposed strategy; at this second stage, we may proceed as if $\eta_0$ were known.} 
The right-hand side of \eqref{hm:exp.qlm-2} is valid for the degenerate $\Sig$ as soon as $\wh{M}_i$ is positive definite, which is true with probability tending to $1$ for $n\to\infty$ under Assumption \ref{hm:A_iden}; 
specifically, one may go through the above calculations to obtain \eqref{hm:exp.qlm-2} first by replacing $\Sig$ by $\Sig+m^{-1} I_{p_\vp}$ for $m>0$ and then applying Scheff\'{e}'s lemma for $m\to\infty$.

Taking the product of \eqref{hm:exp.qlm-2} for $i=1,\dots,N$ and then taking the logarithm, we arrive at the following \textit{second-stage marginal quasi-likelihood} for estimating $\vt_2$:
\begin{align}
    \mbbh_{2}(\vt_2) 
    &:= \sum_i \log \phi_{p_{\vp}}\left(\wh{M}_i^{-1} \wh{v}_i ;\, \mu,\, \wh{M}_i^{-1}+\Sig\right)
    \label{hm:joint.qlm-2} \\
    &= -\frac{p_{\vp} N}{2}\log(2\pi) - \frac12 \sum_i \log\det\big(\wh{M}_i^{-1} + \diag(O,\Sig_r)\big)
    \nn\\
    &{}\qquad 
    -\frac12 \sum_i (\wh{M}_i^{-1}+\diag(O,\Sig_r))^{-1}
    \left[
    \left( \wh{M}_i^{-1} \wh{v}_i - (\vp_f,0)^\top  \right)^{\otimes 2}
    \right].
    \nn
\end{align}
Then, we define an estimator of $\vt_2$ 
by any element
\begin{equation}\label{hm:num.opt-2}
    \wh{\vt}_2=(\wh{\vp}_f,\wh{\Sig}_r) \in \argmax \mbbh_{2}.
\end{equation}
We note that $\mbbh_2(\vt_2)$ is formally the same as in the no-truncated version of the random function $\mathbf{W}_n(X_{i,n},\bm{\mu},\bm{\Omega})$ in \cite[Eq.(31)]{DelGenLar18-2}. 

Let
\begin{align}
    M_i=M_{i,T} := \tau_i \int_0^T \frac{a_i^{\otimes 2}(t)}{S_i(t)}dt,
    \qquad 
    v_i=v_{i,T}:=\int_0^T \frac{a_i(t)}{S_i(t)} dY_i(t).
    \label{hm:M&v_def}
\end{align}
Moreover, let
\begin{align}
    \mbby_{2,0}(\vt_2) &:= \E\left[
    \log\left(
    \frac{
    \phi_{p_{\vp}}\left(M_{1}^{-1} v_{1} ;\, (\vp_f,0),\, M_{1}^{-1}+\diag(O,\Sig_{r})\right)
    }{
    \phi_{p_{\vp}}\left(M_{1}^{-1} v_{1} ;\, (\vp_{f,0},0),\, M_{1}^{-1}+\Sig_0\right)
    }
    \right)\right].
\end{align}
Then, we assume the following.

\begin{ass}[Non-degeneracy of drift coefficient and identifiability]
\label{hm:A_drift}~
\begin{enumerate}
    \item $\ds{\max_{i\le N} \left(\lam_{\min}(M_i)^{-1} + \lam_{\min}(\wh{M}_i)^{-1}\right)=O_p(1)}$.
    
    \item There exists a constant $\mathsf{c}_2>0$ such that $\mbby_{2,0}(\vt_2) \le - \mathsf{c}_2 |\vt_2-\vt_{2,0}|^2$ for every $\vt_2$.
\end{enumerate}
\end{ass}

The non-degeneracy condition, \hmrev{the first item in Assumption \ref{hm:A_drift}}, is a technical one. It would be possible to simplify it similarly to the proof of \cite[Lemma 5.2]{EguMas19} to handle the inverse of $\wh{\tau}_i$. However, it cannot be done for free, and careful considerations are necessary.

As in \cite{DelGenLar18-2}, we introduce the following notation:
\begin{align}
    B_{i,T}(\Sig_r) &= \big(M_{i,T}^{-1}+\diag(0,\Sigma_r)\big)^{-1}, \nn\\
    A_{i,T}(\vp_f,\Sig_r) &= B_{i,T}(\Sig_r)
    \big(M_{i,T}^{-1} U_{i,T} - (\vp_f,0) \big),
    \nn
\end{align}
where
\begin{equation}
    U_{i,T} := \int_0^T \frac{a_i(t)}{S_i(t)}dY_i(t).
\end{equation}
Finally, let $\mcj_{2,0}\in\mbbr^{q_\vp}$, where $q_\vp:=p_\vp + p_\vp (p_\vp +1)/2$, denote the covariance matrix associated with the random variable
\begin{equation}
    \begin{pmatrix}
        \tau_1^{-1} A_{1,T}(\vp_f,\Sig_r) \\[2mm]
        \ds{\mathrm{vech}\left(\frac12 \tau_1^{-1} A_{1,T}(\vp_f,\Sig_r)^{\otimes 2} - B_{1,T}(\Sig_r) \right)}
    \end{pmatrix},
\end{equation}
where $\mathrm{vech}$ denotes the vector-half operator that stacks the lower triangular half into a single vector of length $q_\vp$.

Then, we have the following result.

\begin{thm}\label{hm:thm_drif.param}
Suppose that the setup described in Section \ref{hm:sec_setup.assump} and Assumption \ref{hm:A_drift} hold, and suppose that $\mcj_{2,0}$ is invertible. Then, we have
\begin{equation}
\sqrt{N}(\wh{\vt}_2 -\vt_{2,0}) \cil N_{q_{\vp}}\left(0,\,\mcj_{2,0}^{-1}\right).
\nonumber
\end{equation}
\end{thm}

\begin{rem}
\label{hm:rem_re-2}
    As was mentioned in Remark \ref{hm:rem_diff-1}, it is not essential that the random effect parameters in the drift coefficients take the multiplicative form $\tau_i \vp_i$ for the mutually independent $\tau_i$ and $\vp_i$.
    Because of the stepwise nature of our inference procedure, we could flexibly modify this point in several ways; for example, as in \cite{delattre2018parametric} and \cite{DelGenLar18-2}, we could incorporate $\vp_i$ instead of $\tau_i \vp_i$ with the conditional probability structure $\mcl(\vp_i|\tau_i)=N_{p_{\vp}}(\mu,\tau_i \Sig)$. The modifications are in principle straightforward, and we will not delve into the details here.
\end{rem}

\begin{rem}
    Because of the i.i.d. nature of $Y_1(\cdot), \dots, Y_N(\cdot)$, it would make sense to estimate $\Sig_r$ through the empirical centering of $\wh{M}_1^{-1} \wh{v}_1, \dots, \wh{M}_N^{-1} \wh{v}_N$: letting $\overline{\wh{M}^{-1} \wh{v}} := N^{-1} \sumi \wh{M}_i^{-1} \wh{v}_i$, we may estimate $\Sig_r$ by maximizing
    \begin{equation}
    \Sig_r \mapsto \sumi
    \log \phi_{p_{\vp}}\left(\wh{M}_i^{-1} \wh{v}_i - \overline{\wh{M}^{-1} \wh{v}} ;\, 0,\, \wh{M}_i^{-1}+\diag(0,\Sig_r)\right).
    \end{equation}
\end{rem}

\hmrev{
\begin{rem}
    To the best of the authors' knowledge, the asymptotic efficiency (in the sense of Haj\'{e}k--Le Cam; see, for example, \cite[Section 8]{vdV98}) for SDEs with mixed effects has not yet been addressed in the literature. The asymptotic efficiency issue for our estimators is challenging, not only because the exact likelihood is intractable but also because the stepwise nature of the estimation procedure renders the analysis of the joint distribution of the observed data even more complicated and obscured. Although the obtained rate of convergence is expected to be optimal, the asymptotic efficiency remains far from clear, except for $\widehat{\theta}_\tau$. We leave this challenging issue for future research.
\end{rem}
}

\subsection{Connections to existing results}

It is worth noting that the convergence rates established in our theoretical analysis are in line with several well-known results in the literature, even though these results were originally obtained under (quasi-)maximum likelihood estimation procedures, which is not exactly the case in our framework. In classical nonlinear mixed-effects models, it is established that, when $N$ and $n$ both tend to infinity fixed effects converge at rate $\sqrt{Nn}$, while parameters governing the distribution of random effects converge at rate $\sqrt{N}$, where $N$ denotes the number of individuals and $n$ the number of observations per individual \cite{nie2007convergence}. In more specialized settings, which correspond to particular cases of our model, it has been shown that, for $T$ fixed and $N, n$ tending to infinity, fixed effects in the diffusion coefficient also converge at rate $\sqrt{Nn}$, while fixed effects in the drift and the parameters of the random effects distributions converge at rate $\sqrt{N}$ \cite{delattre2013maximum, delattre2015estimation, DelGenLar18-2, delattre2018parametric}. These findings are consistent with our results and further supported by the structure of the asymptotic covariance matrix for the estimators of the diffusion parameters, which reveals the same level of precision comparable to that obtained if the individual-level random effects were observed.


\section{Numerical experiments}
\label{sec:simus}

We simulate data at discrete time points within a given time interval $[0,T]$ according to the following mixed-effects SDE:

\underline{\it Model 1:}

\begin{equation*}
\begin{split}
Y_i(t) & = Y_i(0) - \tau_i \left( \int_0^t \frac{\varphi_f + \varphi_{r,i}}{\sqrt{1+Y_i(s)^2}} ds \right) + \sqrt{\tau_i} \int_0^t \exp(\eta \, t/2) dw_i(s)\\
\tau_i & \sim \textrm{LogNormal}(\alpha, \lambda), \; \varphi_{r,i} \sim_{i.i.d.} N_1(0,\sigma_r^2), 
\end{split}
\end{equation*}

\medskip

\underline{\it Model 2:}

\begin{equation*}
\begin{split}
Y_i(t) & = Y_i(0) - \tau_i \left(\int_0^t \frac{\varphi_{f}^{(1)} Y_i(s) + \varphi_{f}^{(2)}}{\sqrt{1+Y_i(s)^2}} ds + \int_0^t \frac{\varphi_{r,i} Y_i(s)}{\sqrt{1+Y_i(s)^2}}ds \right)  \, + \\
& \sqrt{\tau_i} \int_0^t \exp(\eta \,\textrm{atan}(Y_i(s))) dw_i(s)\\
\tau_i & \sim \textrm{Weibull}(\alpha,\lambda), \; \varphi_{r,i} \sim_{i.i.d.} N_1(0,\sigma_r^2), 
\end{split}
\end{equation*}

\medskip

\underline{\it Model 3:}

\tcr{
\begin{equation*}
\begin{split}
Y_i(t) & = Y_i(0) + \tau_i \left(\int_0^t (\varphi_{f}^{(1)} Y_i(s) + \varphi_{f}^{(2)}) ds + \int_0^t \varphi_{r,i} Y_i(s)ds \right) \, \\
&{}\qquad + \sqrt{\tau_i} \int_0^t dw_i(s)\\
\tau_i & \sim \textrm{Lognormal}(\alpha, \lambda), \; \; \varphi_{r,i} \sim_{i.i.d.} N_1(0,\sigma_r^2).
\end{split}
\end{equation*}
}

\medskip

\tcr{Note that the Ornstein-Uhlenbeck model (\textit{Model 3}) does not satisfy the assumptions made to derive the theoretical results above and that its diffusion coefficient does not contain any fixed effects to be estimated, as this would lead to identifiability issues. The first step of the estimation procedure in \textit{Model 3} is therefore limited to estimating $\theta_{\tau}$.} 

For each model, data sets are simulated according to different scenarios that combine several numbers of trajectories $N$ (in $\{\tcr{20, 100}, 200,500\}$), several values for $T$ (in $\{5,10\}$) and several time steps $h$ (in $\{0.005, 0.001, \tcr{0.01}\}$) between consecutive observations leading to different numbers of observations per trajectory $n$. Each observation path is simulated according to its Euler scheme with a very small time step $\delta=0.0001$. In each scenario, $500$ datasets are simulated and the model parameters are estimated using the procedures described above. To illustrate the theoretical results obtained above, the parameters $\theta_{\tau}$ are also estimated from the simulated $\tau_i$'s as if they were directly observed. 
The true parameter values used in this numerical study are as follows.
\begin{itemize}
\item {\it Model 1}: $\eta = 0.5$, $\alpha = -0.7$, $\lambda = 0.7$, $\varphi_f = 2$, $\sigma_r^2 = 1$, 
\item {\it Model 2}: $\eta = 0.5$, $\alpha = 1$, $\lambda = 0.6$, $\varphi_f^{(1)} = 2$, $\varphi_f^{(2)} = 1$, $\sigma_r^2 = 1$,
\item {\it Model 3}: \tcr{$\alpha = -0.7$, $\lambda = 0.2$, $\varphi_f^{(1)} = 2$, $\varphi_f^{(2)} = 1$, $\sigma_r^2 = 1$.}
\end{itemize}

\begin{figure}[ht]
    \centering

    \begin{subfigure}{0.32\textwidth}
        \centering
        \textit{Model 1}
        \includegraphics[width=\linewidth]{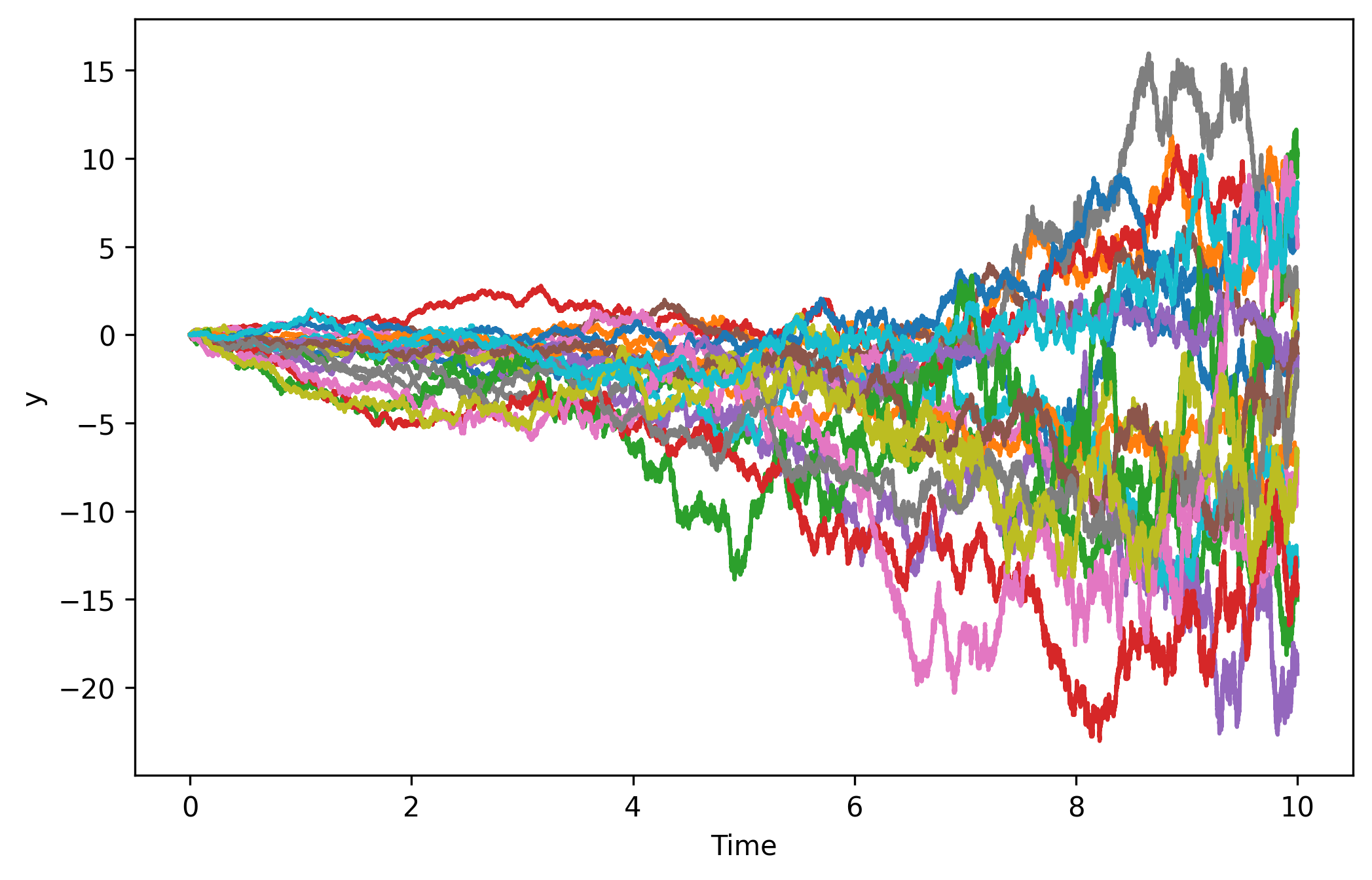}
    \end{subfigure}
    \hfill
    \begin{subfigure}{0.32\textwidth}
        \centering
        \textit{Model 2}
        \includegraphics[width=\linewidth]{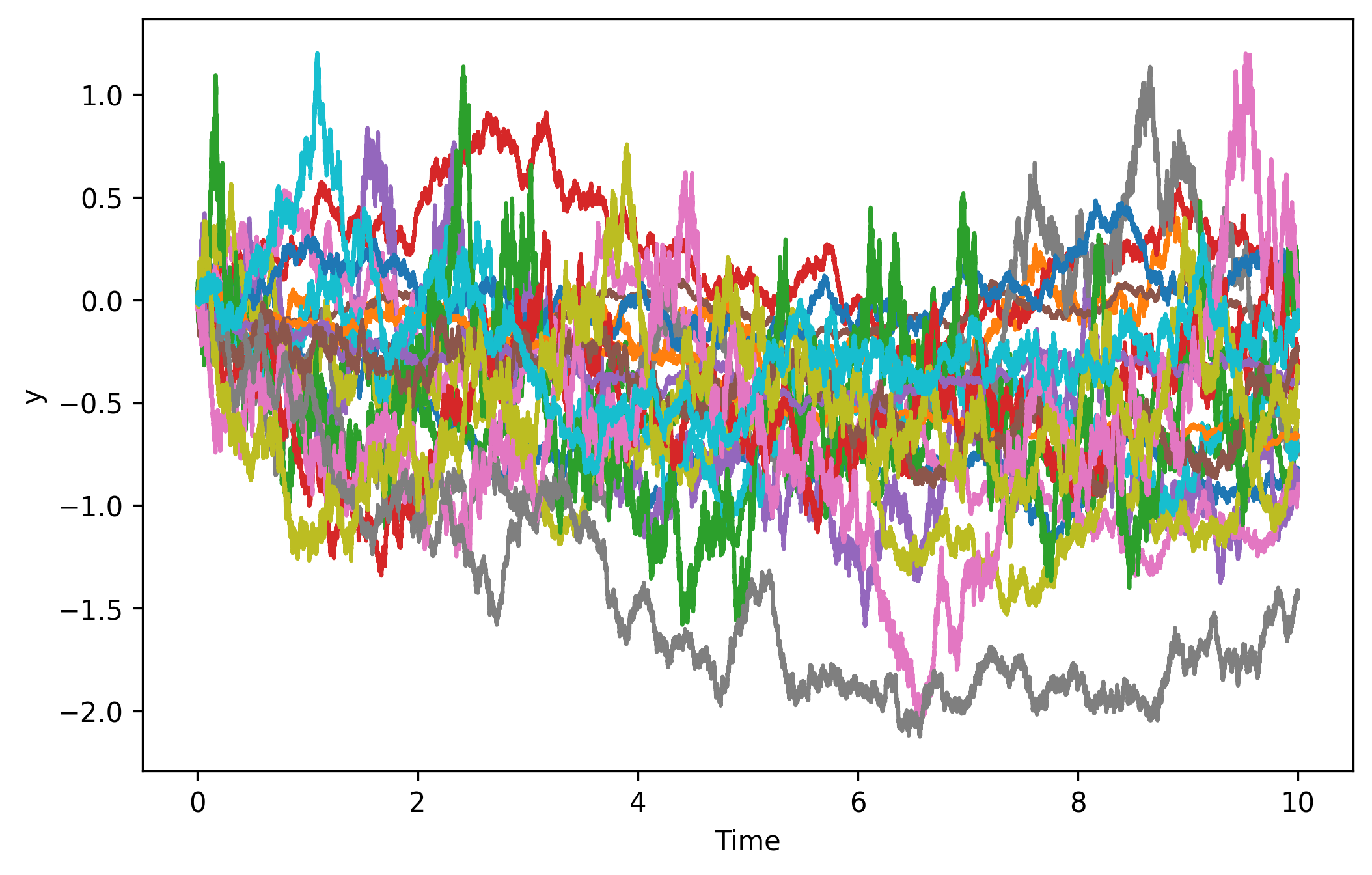}
    \end{subfigure}
    \hfill
    \begin{subfigure}{0.32\textwidth}
        \centering
        \textit{Model 3}
        \includegraphics[width=\linewidth]{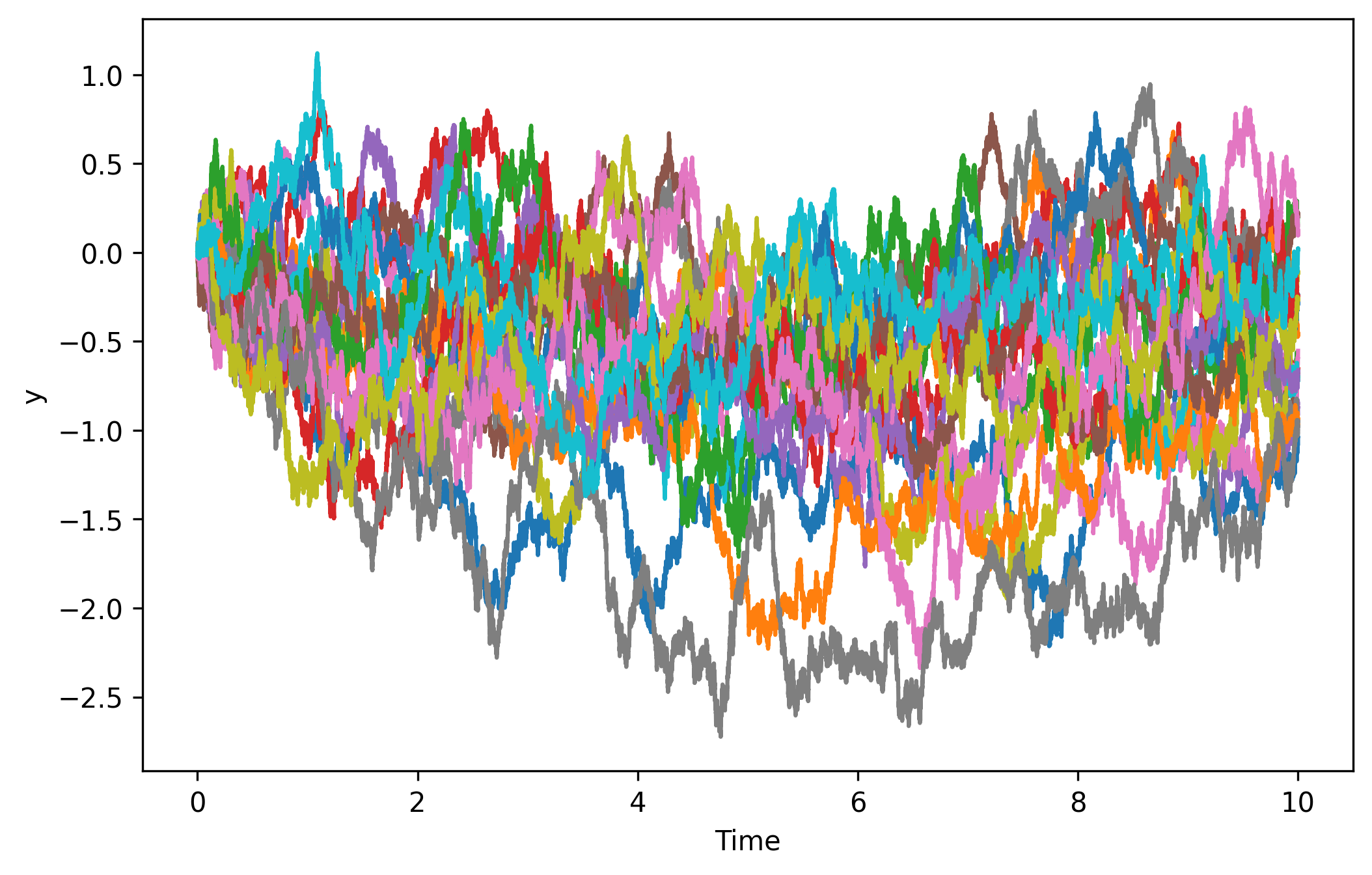}
    \end{subfigure}

    \caption{Simulated trajectories under \textit{Model 1}, \textit{Model 2} and \textit{Model 3}.}
    \label{fig:trajectories}
\end{figure}

\tcr{Figure ~\ref{fig:trajectories} shows twenty simulated trajectories under each of these three models.} The results are depicted in Tables
~\ref{tab:m1-diff-est}-\ref{tab:m1-drift-est} for \textit{Model 1}, Tables~\ref{tab:m2-diff-est}-\ref{tab:m2-drift-est} for \textit{Model 2} and Tables~\ref{tab:m3-diff-est}-\ref{tab:m3-drift-est} for \textit{Model 3}. Each table summarizes the empirical means and standard deviations of the estimates obtained from the $500$ simulated data sets in each of the $(N, n, T)$ configurations.
 
We see from these tables that in each model all parameters are estimated with little or no bias and that their accuracy improves as the sample size increases. \tcr{The procedure performs quite well even for small sample sizes (e.g., $N=20$). Nevertheless, as it is based on an Euler discretization scheme, its reliability may deteriorate when the time step between consecutive observations is too large, and the associated theoretical guarantees may then no longer apply.} \tcr{Moreover, although {\it Model 3} does not fall within the class of processes covered by our theoretical assumptions, the estimation procedure remains empirically stable and provides satisfactory results in the simulated setting. This suggests that the proposed methodology may retain good practical performance beyond the theoretical framework considered in the paper.}

In {\it Model 1}, the asymptotic variance of the estimation of the fixed effect of the diffusion coefficient also has an explicit form $2\sqrt{6}/(T\sqrt{Nn})$. The values of the empirical standard deviations of the estimates obtained for $\eta$ correspond well to their expected theoretical values whatever the values of $N$, $n$ and $T$.   

\tcr{For the largest values of $N$ considered in the simulation study ($N=200$ and $N=500$), the empirical standard deviations of the estimates of $\theta_{\tau}$ in all three models are close to those that would be obtained if $\theta_{\tau}$ were estimated from the true realizations of the $\tau_i$. This agreement is consistent with the asymptotic framework considered in the paper, which relies on a sufficiently large number of trajectories and a fine temporal discretization. The other experiments with more challenging sampling regimes indicate that this agreement becomes less pronounced when moving away from the asymptotic setting (e.g., for smaller values of $N$ or larger observation steps), as expected given the theoretical framework and the use of the Euler approximation.}

Across these 3 models, the good results obtained in this simulation study show that the estimation procedure proposed here is sufficiently generic to handle different distributions for random effects in the diffusion coefficient, fixed and random effects in the drift. The ability to handle any distribution with positive support for the random effect in the diffusion coefficient is a real advance over the state of the art.

\begin{table}[tbp]
\begin{center}
\caption{Simulations under {\bf Model 1}: empirical mean and empirical standard error (in brackets) of $500$ {\bf diffusion parameters} estimates for each combination of $N$, $T$ and $\delta \, (n)$. The true parameter values are $\eta = 0.5$, $\alpha = -0.7$, $\lambda = 0.7$. $\widehat{\eta_y}$, $\widehat{\alpha_y}$ and $\widehat{\lambda_y}$ stand for the estimates computed from the trajectories observations, whereas $\widehat{\alpha_{\tau}}$ and $\widehat{\lambda_{\tau}}$ denote the estimates computed from the simulated $\tau_i$'s.}
\label{tab:m1-diff-est} 
\tcr{
\begin{tabular}{cc|ccc|ccc}
\toprule
$(N,T)$ & $\delta \, (n)$ & \multicolumn{1}{c}{$\widehat{\eta_y}$} & \multicolumn{1}{c}{$\widehat{\alpha_y}$} & \multicolumn{1}{c}{$\widehat{\lambda_y}$} & \multicolumn{1}{c}{$\widehat{\alpha_{\tau}}$} & \multicolumn{1}{c}{$\widehat{\lambda_{\tau}}$}\\
\midrule
\multirow{3}*{(20,5)}
&0.001 (5000) & 0.500 (0.003) & -0.700 (0.180) & 0.683 (0.186) & -0.702 (0.176) & 0.683 (0.148) \\
&0.005 (1000) & 0.498 (0.007) & -0.698 (0.176) & 0.688 (0.138) & -0.702 (0.176) & 0.683 (0.148) \\
&0.01 (500)  & 0.496 (0.010) & -0.694 (0.188) & 0.693 (0.138) & -0.702 (0.176) & 0.683 (0.148) \\
\hdashline
\multirow{3}*{(20,10)}
&0.001 (10000) & 0.500 (0.001) & -0.700 (0.180) & 0.683 (0.186) & -0.702 (0.176) & 0.683 (0.148) \\
&0.005 (2000) & 0.499 (0.002) & -0.702 (0.177) & 0.686 (0.138) & -0.702 (0.176) & 0.683 (0.148) \\
&0.01 (1000)  & 0.499 (0.003) & -0.698 (0.182) & 0.692 (0.127) & -0.702 (0.176) & 0.683 (0.148) \\
\midrule
\multirow{3}*{(100,5)}
&0.001 (5000) & 0.500 (0.001) & -0.698 (0.087) & 0.699 (0.102) & -0.701 (0.09) & 0.698 (0.102) \\
&0.005 (1000) & 0.498 (0.003) & -0.694 (0.091) & 0.704 (0.083) & -0.701 (0.09) & 0.698 (0.102) \\
&0.01 (500)  & 0.496 (0.005) & -0.691 (0.096) & 0.706 (0.083) & -0.701 (0.09) & 0.698 (0.102) \\
\hdashline
\multirow{3}*{(100,10)}
&0.001 (10000) & 0.500 (0.000) & -0.698 (0.087) & 0.698 (0.102) & -0.701 (0.09) & 0.698 (0.102) \\
&0.005 (2000) & 0.499 (0.001) & -0.696 (0.087) & 0.702 (0.082) & -0.701 (0.09) & 0.698 (0.102) \\
&0.01 (1000)  & 0.499 (0.002) & -0.695 (0.09) & 0.703 (0.082) & -0.701 (0.09) & 0.698 (0.102) \\
\midrule
\multirow{3}*{(200,5)}
&0.001 (5000) & 0.500 (0.001) & -0.703 (0.067) & 0.703 (0.074) & -0.704 (0.069) & 0.703 (0.074) \\
&0.005 (1000) & 0.498 (0.002) & -0.701 (0.069) & 0.705 (0.075) & -0.704 (0.069) & 0.703 (0.074) \\
&0.01 (500)  & 0.496 (0.003) & -0.696 (0.069) & 0.709 (0.075) & -0.704 (0.069) & 0.703 (0.074) \\
\hdashline
\multirow{3}*{(200,10)}
&0.001 (10000) & 0.500 (0.000) & -0.703 (0.071) & 0.703 (0.074) & -0.704 (0.069) & 0.703 (0.074) \\
&0.005 (2000) & 0.499 (0.001) & -0.701 (0.068) & 0.704 (0.074) & -0.704 (0.069) & 0.703 (0.074) \\
&0.01 (1000)  & 0.499 (0.001) & -0.699 (0.071) & 0.706 (0.075) & -0.704 (0.069) & 0.703 (0.074) \\
\midrule
\multirow{3}*{(500,5)}
&0.001 (5000) & 0.500 (0.001) & -0.705 (0.055) & 0.704 (0.069) & -0.704 (0.055) & 0.703 (0.068) \\
&0.005 (200) & 0.498 (0.001) & -0.700 (0.061) & 0.707 (0.069) & -0.704 (0.055) & 0.703 (0.068) \\
&0.01 (100)  & 0.496 (0.002) & -0.695 (0.062) & 0.710 (0.070) & -0.704 (0.055) & 0.703 (0.068) \\
\hdashline
\multirow{3}*{(500,10)}
&0.001 (10000) & 0.500 (0.000) & -0.705 (0.056) & 0.703 (0.069) & -0.704 (0.055) & 0.703 (0.068) \\
&0.005 (2000) & 0.499 (0.001) & -0.702 (0.057) & 0.705 (0.069) & -0.704 (0.055) & 0.703 (0.068) \\
&0.01 (1000)  & 0.499 (0.001) & -0.702 (0.054) & 0.706 (0.069) & -0.704 (0.055) & 0.703 (0.068) \\
\bottomrule
\end{tabular}
}
\end{center}
\end{table}

\begin{table}[tbp]
\begin{center}
\caption{Simulations under {\bf Model 1}: empirical mean and empirical standard error (in brackets) of $500$ {\bf drift} parameters estimates for each combination of $N$, $T$ and $\delta \, (n)$. The true parameter values are $\varphi_f = 2$, $\sigma_r^2 = 1$.}
\label{tab:m1-drift-est}
\tcr{
\begin{tabular}{cc|c|cc}
\toprule
$(N,T)$ & $\delta \, (n)$ & \multicolumn{1}{c}{$\widehat{\varphi_f}$} &
\multicolumn{1}{c}{$\widehat{\sigma_r^2}$} \\
\midrule
\multirow{3}{*}{(20,5)}
&0.001 (5000) & 2.295 (2.043) & 0.583 (2.230)\\
&0.005 (1000) & 2.271 (1.132) & 0.520 (2.036)\\
&0.01 (500) & 0.596 (37.305) & 0.483 (2.068)\\
\hdashline
\multirow{3}{*}{(20,10)}
&0.001 (10000) & 2.457 (5.771) & 0.461 (2.435)\\
&0.005 (2000) & 3.003 (9.531) & 0.398 (2.082)\\
&0.01 (1000) & 2.102 (5.541) & 0.445 (2.088)\\
\midrule
\multirow{3}{*}{(100,5)}
&0.001 (5000) & 2.280 (0.180) & 0.750 (0.470)\\
&0.005 (1000) & 2.280 (0.220) & 0.720 (0.460)\\
&0.01 (500) & 2.250 (0.170) & 0.670 (0.450)\\
\hdashline
\multirow{3}{*}{(100,10)}
&0.001 (10000) & 2.280 (0.170) & 0.770 (0.450)\\
&0.005 (2000) & 2.270 (0.170) & 0.750 (0.440)\\
&0.01 (1000) & 2.270 (0.170) & 0.720 (0.440)\\
\midrule
\multirow{3}{*}{(200,5)}
&0.001 (5000) & 2.270 (0.120) & 0.740 (0.320)\\
&0.005 (1000) & 2.250 (0.120) & 0.710 (0.310)\\
&0.01 (500) & 2.240 (0.120) & 0.670 (0.310)\\
\hdashline
\multirow{3}{*}{(200,10)}
&0.001 (10000) & 2.260 (0.120) & 0.760 (0.310)\\
&0.005 (2000) & 2.260 (0.120) & 0.730 (0.300)\\
&0.01 (1000) & 2.250 (0.120) & 0.710 (0.300)\\
\midrule
\multirow{3}{*}{(500,5)}
&0.001 (5000) & 2.260 (0.080) & 0.730 (0.200)\\
&0.005 (1000) & 2.250 (0.070) & 0.700 (0.200)\\
&0.01 (500) & 2.240 (0.070) & 0.660 (0.190)\\
\hdashline
\multirow{3}{*}{(500,10)}
&0.001 (10000) & 2.260 (0.080) & 0.750 (0.200)\\
&0.005 (2000) & 2.250 (0.080) & 0.720 (0.190)\\
&0.01 (1000) & 2.250 (0.080) & 0.700 (0.190)\\
\bottomrule
\end{tabular}
}
\end{center}
\end{table}

\begin{table}[tbp]
\centering
\caption{Simulations under {\bf Model 2}: empirical mean and empirical standard error (in brackets) of $500$ {\bf diffusion parameters} estimates for each combination of $N$, $T$ and $\delta \, (n)$. The true parameter values are $\eta = 0.5$, $\alpha = 1$, $\lambda = 0.6$. $\widehat{\eta_y}$, $\widehat{\lambda_y}$ and $\widehat{\alpha_y}$ stand for the estimates computed from the trajectories observations, whereas $\widehat{\lambda_{\tau}}$ and $\widehat{\alpha_{\tau}}$ denote the estimates computed from the simulated $\tau_i$'s.}
\label{tab:m2-diff-est}
\tcr{
\begin{tabular}{cc|rrr|rr}
\toprule
$(N,T)$ & $\delta \, (n)$ & \multicolumn{1}{c}{$\widehat{\eta_y}$} & \multicolumn{1}{c}{$\widehat{\alpha_y}$} & \multicolumn{1}{c}{$\widehat{\lambda_y}$} & \multicolumn{1}{c}{$\widehat{\alpha_{\tau}}$} & \multicolumn{1}{c}{$\widehat{\lambda_{\tau}}$}\\
\midrule
\multirow{3}*{(20,5)}
&0.001 (5000) & 0.500 (0.008) & 1.057 (0.189) & 0.620 (0.139) & 1.058 (0.189) & 0.619 (0.139)\\
&0.005 (1000) & 0.498 (0.016) & 1.059 (0.192) & 0.617 (0.138) & 1.058 (0.189) & 0.619 (0.139)\\
&0.01 (500) & 0.497 (0.024) & 1.059 (0.190) & 0.613 (0.136) & 1.058 (0.189) & 0.619 (0.139)\\
\hdashline
\multirow{3}*{(20,10)}
&0.001 (10000) & 0.500 (0.005) & 1.058 (0.189) & 0.618 (0.137) & 1.058 (0.189) & 0.619 (0.139)\\
&0.005 (2000) & 0.497 (0.011) & 1.061 (0.190) & 0.616 (0.137) & 1.058 (0.189) & 0.619 (0.139)\\
&0.01 (1000) & 0.497 (0.016) & 1.061 (0.189) & 0.614 (0.136) & 1.058 (0.189) & 0.619 (0.139)\\
\midrule
\multirow{3}*{(100,5)}
&0.001 (5000) & 0.500 (0.003) & 1.014 (0.080) & 0.610 (0.071) & 1.013 (0.081) & 0.610 (0.070)\\
&0.005 (1000) & 0.498 (0.008) & 1.014 (0.082) & 0.608 (0.073) & 1.013 (0.081) & 0.610 (0.070)\\
&0.01 (500) & 0.497 (0.011) & 1.017 (0.085) & 0.604 (0.070) & 1.013 (0.081) & 0.610 (0.070)\\
\hdashline
\multirow{3}*{(100,10)}
&0.001 (10000) & 0.500 (0.002) & 1.013 (0.082) & 0.609 (0.071) & 1.013 (0.081) & 0.610 (0.070)\\
&0.005 (2000) & 0.498 (0.005) & 1.014 (0.082) & 0.607 (0.070) & 1.013 (0.081) & 0.610 (0.070)\\
&0.01 (1000) & 0.497 (0.007) & 1.017 (0.083) & 0.606 (0.073) & 1.013 (0.081) & 0.610 (0.070)\\
\midrule
\multirow{3}*{(200,5)}
&0.001 (5000) & 0.500 (0.002) & 1.004 (0.062) & 0.605 (0.053) & 1.005 (0.063) & 0.605 (0.053)\\
&0.005 (1000) & 0.498 (0.006) & 1.006 (0.065) & 0.603 (0.054) & 1.005 (0.063) & 0.605 (0.053)\\
&0.01 (500) & 0.497 (0.008) & 1.008 (0.063) & 0.600 (0.053) & 1.005 (0.063) & 0.605 (0.053)\\
\hdashline
\multirow{3}*{(200,10)}
&0.001 (10000) & 0.500 (0.002) & 1.005 (0.062) & 0.605 (0.054) & 1.005 (0.063) & 0.605 (0.053)\\
&0.005 (2000) & 0.498 (0.004) & 1.007 (0.063) & 0.603 (0.055) & 1.005 (0.063) & 0.605 (0.053)\\
&0.01 (1000) & 0.497 (0.005) & 1.009 (0.062) & 0.600 (0.053) & 1.005 (0.063) & 0.605 (0.053)\\
\midrule
\multirow{3}*{(500,5)}
&0.001 (5000) & 0.500 (0.002) & 1.000 (0.047) & 0.605 (0.041) & 1.000 (0.047) & 0.606 (0.042)\\
&0.005 (1000) & 0.498 (0.003) & 1.000 (0.049) & 0.604 (0.042) & 1.000 (0.047) & 0.606 (0.042)\\
&0.01 (500) & 0.496 (0.005) & 1.003 (0.048) & 0.599 (0.038) & 1.000 (0.047) & 0.606 (0.042)\\
\hdashline
\multirow{3}*{(500,10)}
&0.001 (10000) & 0.500 (0.001) & 1.001 (0.047) & 0.606 (0.042) & 1.000 (0.047) & 0.606 (0.042)\\
&0.005 (2000) & 0.498 (0.002) & 1.001 (0.046) & 0.604 (0.043) & 1.000 (0.047) & 0.606 (0.042)\\
&0.01 (1000) & 0.496 (0.003) & 1.003 (0.047) & 0.599 (0.037) & 1.000 (0.047) & 0.606 (0.042)\\
\bottomrule
\end{tabular}
}
\end{table}

\begin{table}[tbp]
\centering
\caption{Simulations under {\bf Model 2}: empirical mean and empirical standard error (in brackets) of $500$ {\bf drift} parameters estimates for each combination of $N$, $T$ and $\delta \, (n)$. The true parameter values are $\varphi_f^{(1)} = 2$, $\varphi_f^{(2)} = 1$, $\sigma_r^2 = 1$.}
\label{tab:m2-drift-est}
\tcr{
\begin{tabular}{rc|rrr}
\toprule
$(N,T)$ & $\delta \, (n)$ & \multicolumn{1}{c}{$\widehat{\varphi_f^{(1)}}$} & \multicolumn{1}{c}{$\widehat{\varphi_f^{(2)}}$} & \multicolumn{1}{c}{$\widehat{\sigma_r^2}$} \\
\midrule
\multirow{3}*{(20,5)}
&0.001 (5000) & 2.411 (0.660) & 1.233 (0.298) & 0.837 (0.832)\\
&0.005 (1000) & 2.415 (0.598) & 1.234 (0.292) & 0.830 (0.828)\\
&0.01 (500) & 2.414 (0.554) & 1.233 (0.287) & 0.821  (0.815)\\
\hdashline
\multirow{3}*{(20,10)}
&0.001 (10000) & 2.299 (0.429) & 1.203 (0.217) & 0.888 (0.632)\\
&0.005 (2000) & 2.302 (0.422) & 1.204 (0.217) & 0.886 (0.631)\\
&0.01 (1000) & 2.299 (0.425) & 1.203 (0.217) & 0.882 (0.627)\\
\midrule
\multirow{3}*{(100,5)}
&0.001 (5000) & 2.409 (0.211) & 1.227 (0.115) & 1.019 (0.381)\\
&0.005 (1000) & 2.407 (0.211) & 1.226 (0.115) & 1.012 (0.378)\\
&0.01 (500) & 2.402 (0.211) & 1.225 (0.115) & 1.004 (0.375)\\
\hdashline
\multirow{3}*{(100,10)}
&0.001 (10000) & 2.293 (0.168) & 1.202 (0.090) & 1.030 (0.283)\\
&0.005 (2000) & 2.292 (0.167) & 1.202 (0.091) & 1.027 (0.282)\\
&0.01 (1000) & 2.291 (0.168) & 1.202 (0.091) & 1.023 (0.281)\\
\midrule
\multirow{3}*{(200,5)}
&0.001 (5000) & 2.407 (0.152) & 1.217 (0.080) & 1.037 (0.273)\\
&0.005 (1000) & 2.405 (0.152) & 1.216 (0.080) & 1.031 (0.272)\\
&0.01 (500) & 2.400 (0.152) & 1.215 (0.080) & 1.023 (0.270)\\
\hdashline
\multirow{3}*{(200,10)}
&0.001 (10000) & 2.286 (0.117) & 1.194 (0.062) & 1.037 (0.205)\\
&0.005 (2000) & 2.286 (0.117) & 1.194 (0.062) & 1.034 (0.205)\\
&0.01 (1000) & 2.285 (0.117) & 1.194 (0.062) & 1.031 (0.204)\\
\midrule
\multirow{3}*{(500,5)}
&0.001 (5000) & 2.412 (0.095) & 1.218 (0.051) & 1.051 (0.172)\\
&0.005 (1000) & 2.409 (0.095) & 1.218 (0.051) & 1.046 (0.172)\\
&0.01 (500) & 2.405 (0.095) & 1.217 (0.051) & 1.038 (0.170)\\
\hdashline
\multirow{3}*{(500,10)}
&0.001 (10000) & 2.288 (0.079) & 1.194 (0.041) & 1.042 (0.135)\\
&0.005 (2000) & 2.287 (0.079) & 1.195 (0.041) & 1.039 (0.134)\\
&0.01 (1000) & 2.287 (0.079) & 1.195 (0.041) & 1.036 (0.134)\\
\bottomrule
\end{tabular}
}
\end{table}

\begin{table}[tbp]
\begin{center}
\caption{Simulations under {\bf Model 3}: empirical mean and empirical standard error (in brackets) of $500$ {\bf diffusion parameters} estimates for each combination of $N$, $T$ and $\delta \, (n)$. The true parameter values are $\alpha = -0.7$, $\lambda = 0.2$. $\widehat{\eta_y}$, $\widehat{\alpha_y}$ and $\widehat{\lambda_y}$ stand for the estimates computed from the trajectories observations, whereas $\widehat{\alpha_{\tau}}$ and $\widehat{\lambda_{\tau}}$ denote the estimates computed from the simulated $\tau_i$'s.}
\label{tab:m3-diff-est} 
\tcr{
\begin{tabular}{rc|rrrr}
\toprule
$(N,T)$ & $n$
& \multicolumn{1}{c}{$\widehat{\alpha}$}
& \multicolumn{1}{c}{$\widehat{\alpha_t}$}
& \multicolumn{1}{c}{$\widehat{\lambda}$}
& \multicolumn{1}{c}{$\widehat{\lambda_t}$} \\
\midrule
\multirow{3}*{(20,5)}
&0.001 (5000) & -0.703 (0.112) & -0.702 (0.105) & 0.209 (0.267) & 0.215 (0.085)\\
&0.005 (1000) & -0.686 (0.378) & -0.702 (0.105) & 0.170 (1.250) & 0.215 (0.085)\\
&0.01 (500)  & -0.712 (0.116) & -0.702 (0.105) & 0.232 (0.107) & 0.215 (0.085)\\
\hdashline
\multirow{3}*{(20,10)}
&0.001 (10000) & -0.694 (0.112) & -0.702 (0.105) & 0.236 (0.197) & 0.215 (0.085)\\
&0.005 (2000)  & -0.670 (0.603) & -0.702 (0.105) & 0.170 (1.657) & 0.215 (0.085)\\
&0.01 (1000)  & -0.699 (0.114) & -0.702 (0.105) & 0.258 (0.251) & 0.215 (0.085)\\
\midrule
\multirow{3}*{(100,5)}
&0.001 (5000) & -0.707 (0.158) & -0.674 (0.489) & 0.222 (0.182) & 0.200 (0.663)\\
&0.005 (1000) & -0.694 (0.107) & -0.674 (0.489) & 0.228 (0.170) & 0.200 (0.663)\\
&0.01 (500)  & -0.695 (0.107) & -0.674 (0.489) & 0.239 (0.089) & 0.200 (0.663)\\
\hdashline
\multirow{3}*{(100,10)}
&0.001 (10000) & -0.691 (0.110) & -0.674 (0.489) & 0.266 (0.164) & 0.200 (0.663)\\
&0.005 (2000)  & -0.688 (0.114) & -0.674 (0.489) & 0.288 (0.200) & 0.200 (0.663)\\
&0.01 (1000)  & -0.687 (0.104) & -0.674 (0.489) & 0.301 (0.215) & 0.200 (0.663)\\
\midrule
\multirow{3}*{(200,5)}
&0.001 (5000) & -0.699 (0.108) & -0.699 (0.107) & 0.229 (0.089) & 0.225 (0.083)\\
&0.005 (1000) & -0.702 (0.127) & -0.699 (0.107) & 0.224 (0.170) & 0.225 (0.083)\\
&0.01 (500)  & -0.663 (0.896) & -0.699 (0.107) & 0.091 (3.278) & 0.225 (0.083)\\
\hdashline
\multirow{3}*{(200,10)}
&0.001 (10000) & -0.694 (0.117) & -0.699 (0.107) & 0.263 (0.159) & 0.225 (0.083)\\
&0.005 (2000)  & -0.691 (0.119) & -0.699 (0.107) & 0.294 (0.175) & 0.225 (0.083)\\
&0.01 (1000)  & -0.689 (0.126) & -0.699 (0.107) & 0.309 (0.203) & 0.225 (0.083)\\
\midrule
\multirow{3}*{(500,5)}
&0.001 (5000) & -0.696 (0.098) & -0.695 (0.103) & 0.224 (0.081) & 0.225 (0.094)\\
&0.005 (1000) & -0.701 (0.098) & -0.695 (0.103) & 0.231 (0.085) & 0.225 (0.094)\\
&0.01 (500)  & -0.704 (0.109) & -0.695 (0.103) & 0.240 (0.084) & 0.225 (0.094)\\
\hdashline
\multirow{3}*{(500,10)}
&0.001 (10000) & -0.689 (0.111) & -0.695 (0.103) & 0.276 (0.132) & 0.225 (0.094)\\
&0.005 (2000)  & -0.700 (0.184) & -0.695 (0.103) & 0.295 (0.250) & 0.225 (0.094)\\
&0.01 (1000)  & -0.705 (0.134) & -0.695 (0.103) & 0.321 (0.313) & 0.225 (0.094)\\
\bottomrule
\end{tabular}
}
\end{center}
\end{table}

\begin{table}[tbp]
\begin{center}
\caption{Simulations under {\bf Model 3}: empirical mean and empirical standard error (in brackets) of $500$ {\bf drift parameters} estimates for each combination of $N$, $T$ and $\delta \, (n)$. The true parameter values are $\varphi_f^{(1)} = 2$, $\varphi_f^{(2)} = 1$ and $\sigma_r^{2} = 1$.}
\label{tab:m3-drift-est} 
\tcr{
\begin{tabular}{rc|rrr}
\toprule
$(N,T)$ & $\delta \, (n)$
& \multicolumn{1}{c}{$\widehat{\varphi_f^{(1)}}$}
& \multicolumn{1}{c}{$\widehat{\varphi_f^{(2)}}$}
& \multicolumn{1}{c}{$\widehat{\sigma_r^{2}}$}\\
\midrule
\multirow{3}*{(20,5)}
&0.001 (5000) & 2.523 (0.446) & 1.288 (0.249) & 0.998 (0.780)\\
&0.005 (1000) & 2.517 (0.445) & 1.285 (0.248) & 0.986 (0.768)\\
&0.01 (500)  & 2.509 (0.446) & 1.282 (0.247) & 0.973 (0.758)\\
\hdashline
\multirow{3}*{(20,10)}
&0.001 (10000) & 2.322 (0.332) & 1.225 (0.170) & 0.993 (0.552)\\
&0.005 (2000)  & 2.319 (0.332) & 1.222 (0.170) & 0.981 (0.540)\\
&0.01 (1000)  & 2.315 (0.332) & 1.220 (0.170) & 0.974 (0.536)\\
\midrule
\multirow{3}*{(100,5)}
&0.001 (5000) & 2.507 (0.189) & 1.271 (0.102) & 1.178 (0.368)\\
&0.005 (1000) & 2.501 (0.189) & 1.268 (0.102) & 1.163 (0.359)\\
&0.01 (500)  & 2.493 (0.189) & 1.265 (0.101) & 1.149 (0.355)\\
\hdashline
\multirow{3}*{(100,10)}
&0.001 (10000) & 2.304 (0.149) & 1.219 (0.074) & 1.116 (0.251)\\
&0.005 (2000)  & 2.300 (0.149) & 1.216 (0.074) & 1.102 (0.245)\\
&0.01 (1000)  & 2.296 (0.149) & 1.213 (0.074) & 1.092 (0.243)\\
\midrule
\multirow{3}*{(200,5)}
&0.001 (5000) & 2.512 (0.133) & 1.263 (0.071) & 1.208 (0.255)\\
&0.005 (1000) & 2.506 (0.133) & 1.260 (0.071) & 1.192 (0.250)\\
&0.01 (500) & 2.499 (0.133) & 1.257 (0.070) & 1.177 (0.246)\\
\hdashline
\multirow{3}*{(200,10)}
&0.001 (10000) & 2.300 (0.105) & 1.212 (0.051) & 1.128 (0.179)\\
&0.005 (2000) & 2.296 (0.105) & 1.208 (0.051) & 1.112 (0.176)\\
&0.01 (1000) & 2.293 (0.105) & 1.206 (0.051) & 1.102 (0.174)\\
\midrule
\multirow{3}*{(500,5)}
&0.001 (5000) & 2.514 (0.084) & 1.263 (0.045) & 1.212 (0.163)\\
&0.005 (1000) & 2.508 (0.084) & 1.260 (0.045) & 1.195 (0.160)\\
&0.01 (500) & 2.501 (0.084) & 1.257 (0.045) & 1.181 (0.157)\\
\hdashline
\multirow{3}*{(500,10)}
&0.001 (10000) & 2.301 (0.069) & 1.212 (0.034) & 1.131 (0.113)\\
&0.005 (2000) & 2.297 (0.069) & 1.208 (0.034) & 1.116 (0.110)\\
&0.01 (1000) & 2.294 (0.069) & 1.206 (0.034) & 1.106 (0.109)\\
\bottomrule
\end{tabular}
}
\end{center}
\end{table}

\section{Application to real data}
\label{sec:appli}

We chose to illustrate our approach using neuronal data from the R package \texttt{mixedsde}. The data set consists of $N=240$ membrane potential trajectories that contain $n=2000$ points spaced by a time step of $h=0.00015$ s. This data set has been the subject of several published analyses (see \textit{e.g.} \cite{dion2019mixedsde}), in which the authors primarily employed Ornstein-Uhlenbeck processes with Gaussian random effects and fixed effects in the drift, along with a constant diffusion coefficient.

Because the theoretical guarantees of our estimation procedure have so far been proven only for models with bounded coefficients, we began our analysis of the neural data by focusing on two models that meet these assumptions, leaving out the Ornstein–Uhlenbeck model, which does not comply with them:

\begin{itemize}
\item \underline{\it Neural data model 1 (N1):}
\begin{align}
Y_i(t) &= Y_i(0) + \tau_i \int_0^t \left(\frac{\varphi_{i1} Y_i(s)}{\sqrt{1+Y_i(s)^2}} + \frac{\varphi_{i2}}{\sqrt{1+Y_i(s)^2}}\right) ds 
\nn\\
&{}\qquad + \sqrt{\tau_i} \int_0^t \exp(\eta\textrm{atan}(Y_i(s))) dw_i(s)  
\end{align}
\item \underline{\it Neural data model 2 (N2):}
\begin{align}
Y_i(t) &= Y_i(0) + \tau_i \int_0^t \left(\frac{\varphi_{i1} Y_i(s)}{\sqrt{1+Y_i(s)^2}} + \frac{\varphi_{i2}}{\sqrt{1+Y_i(s)^2}}\right) ds 
\nn\\
&{}\qquad + \sqrt{\tau_i} \int_0^t \exp(\eta s/2) dw_i(s)    
\end{align}
\end{itemize}
with, for both (N1) and (N2)
$$
\varphi_i = (\varphi_{i1},\varphi_{i2})^\top \sim_{i.i.d.} N_2(\varphi_f,\Sigma)\; , \; \tau_i \sim \mathcal{L}(\theta_{\tau}) \; , \; i=1,\ldots,N.
$$

Before our analysis, we scaled the signal by multiplying the membrane potential measurements by a factor of $200$. This scaling was performed to make visible an inter-trajectory variability that would otherwise be challenging to discern.

In the first step of our approach, we examined the empirical distribution of the estimated parameters in the diffusion coefficients of each trajectory to select the most appropriate underlying distribution. We attempted to fit several distributions: exponential, (generalized) Weibull, gamma, and log-normal. The QQ plots indicated that a generalized Weibull distribution provided the best fit. Its density is given by
$$
f(\tau;\gamma,\alpha,\lambda) = \gamma \left(\frac{\tau - \lambda}{\alpha}\right)^{\gamma-1}\exp\left[-\left(\frac{\tau - \lambda}{\alpha}\right)^{\gamma}\right].
$$
Numerical instabilities due to a non-invertible matrix $\Sigma$ in the second step motivated us to specify the first parameter $\varphi_{i1}$ as a fixed effect and the second $\varphi_{i2}$ as a random effect with variance $\sigma_r^2$. The strength of our approach lies in its ability to incorporate random effects following different distributions in the diffusion coefficient, while allowing both random and fixed effects in the drift term. The estimated parameter values are as follows:
\begin{itemize}
\item (N1): $\widehat{\eta} = -0.075$, $\widehat{\lambda} = 4.123$, $\widehat{\alpha} = 5.093$, $\widehat{\gamma} = 3.437$, $\widehat{\varphi_f} = (-5.017,10.260)^\top$, $\widehat{\sigma_r^2} = 2.444$.
\item (N2): $\widehat{\eta} = -0.273$, $\widehat{\lambda} = 3.683$, $\widehat{\alpha} = 4.546$, $\widehat{\gamma} = 3.463$, $\widehat{\varphi_f} = (-5.595,11.461 )^\top$, $\widehat{\sigma_r^2} = 3.156$.
\end{itemize}

Once all model parameters had been estimated, we simulated 1000 trajectories under the estimated models to graphically assess the fit quality of models (N1) and (N2). The corresponding results are shown in Figures \ref{fig:neurons-vpc-n1} and \ref{fig:neurons-vpc-n2} for models (N1) and (N2), respectively. In each figure, the simulated trajectories (in gray) are displayed together with the 95\% prediction intervals (in green) and the neuronal data (in black). The results appear visually convincing, despite the observation time frame ($T$) being shorter than in the simulations presented in the previous sections. However, a visual comparison alone does not allow us to reliably identify which of these two models offers a superior fit to the data.

\begin{figure}[tbp]
    \centering
    \begin{subfigure}[t]{0.48\linewidth}
        \includegraphics[width=\linewidth]{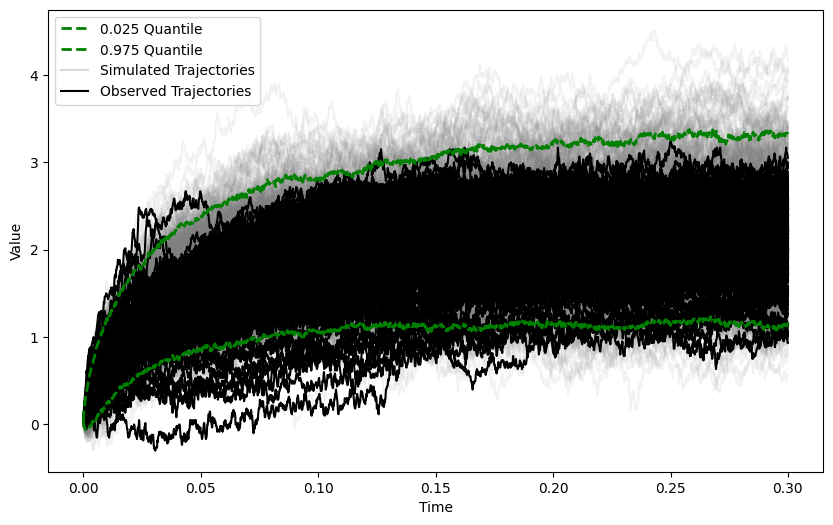}
        \caption{Model (N1).}
        \label{fig:neurons-vpc-n1}
    \end{subfigure}
    \hfill
    \begin{subfigure}[t]{0.48\linewidth}
        \includegraphics[width=\linewidth]{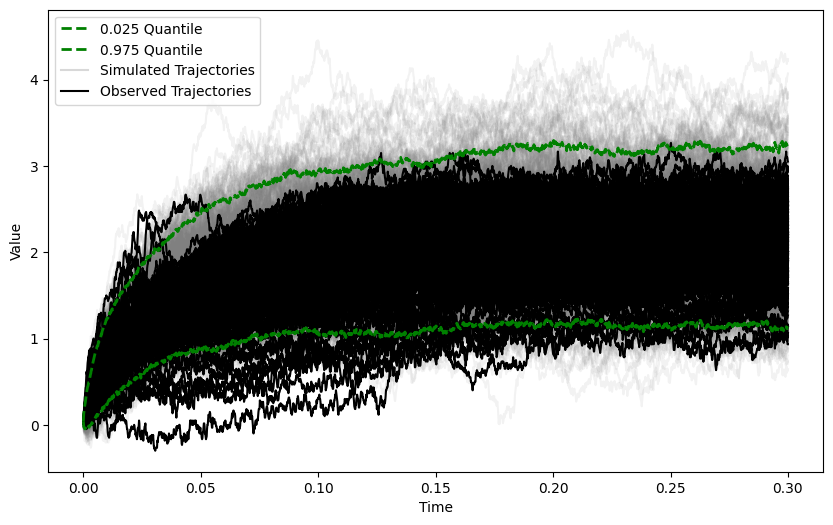}
        \caption{Model (N2).}
        \label{fig:neurons-vpc-n2}
    \end{subfigure}
    \caption{Models (N1) and (N2). Simulated and observed trajectories of membrane potentials. The grey lines represent 1000 simulated trajectories generated from the estimated model, illustrating the expected variability in the system. The green shaded area indicates the 95\% prediction intervals for the simulated trajectories, providing a range within which future observations are expected to fall. The black lines represent the observed neuronal data, highlighting the actual recorded trajectories of membrane potentials. This comparison allows for visual assessment of the model fit and the adequacy of the estimated parameters.}
\end{figure}

Note that we also applied a more general model, consistent with those commonly used in the literature, which includes the Ornstein–Uhlenbeck process as a special case when $\eta = 0$. This model, which incorporates both random and fixed effects in the drift and diffusion coefficients, does not satisfy the assumptions required to establish theoretical properties of the estimators: 
\begin{itemize}
\item \underline{\it Neural data model 3 (N3):}
$$
Y_i(t) = Y_i(0) + \tau_i \int_{0}^{t}(\varphi_{i1} + \varphi_{i2} Y_i(s))ds + \int_{0}^{t} \sqrt{\tau_i(1 + \eta Y_i(s)^2)} dw_i(s)
$$
with 
$$
\varphi_i = (\varphi_{i1},\varphi_{i2})^\top \sim N_2(\varphi_f,\Sigma)\; , \; \tau_i \sim \mathcal{L}(\theta_{\tau}) \; , \; i=1,\ldots,N.
$$
\end{itemize}
Again, a generalized Weibull distribution best fitted the data. The estimated parameter values for model (N3) are $\widehat{\eta} = -0.013$, $\widehat{\lambda} = 3.699$, $\widehat{\alpha} = 4.592$, $\widehat{\gamma} = 3.447$, $\widehat{\varphi_f} = (-4.838,9.770)^\top$, $\widehat{\Sigma} = \begin{pmatrix} 0.039 & 0.028\\0.028 & 2.148\end{pmatrix}$. 

Visual predictive checks are shown in Figure \ref{fig:neurons-vpc-n3}. Although model (N3) lacks formal theoretical guarantees, the results suggest that the estimation procedure performs well and fits the data as well as models (N1) and (N2). A natural next step would be to perform model selection, comparing models (N1), (N2), and (N3); however, this is beyond the scope of the present work.

\begin{figure}[tbp]
    \centering
    \includegraphics[scale=0.6]{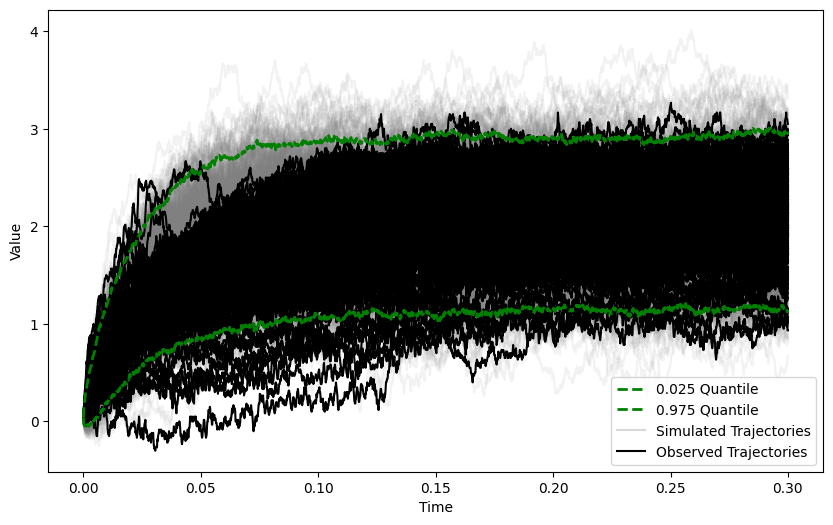}
    \caption{Model (N3). Simulated and observed trajectories of membrane potentials. The grey lines represent 1000 simulated trajectories generated from the estimated model, illustrating the expected variability in the system. The green shaded area indicates the 95\% prediction intervals for the simulated trajectories, providing a range within which future observations are expected to fall. The black lines represent the observed neuronal data, highlighting the actual recorded trajectories of membrane potentials. This comparison allows for visual assessment of the model fit and the adequacy of the estimated parameters.}
    \label{fig:neurons-vpc-n3}
\end{figure}

\tcr{Finally, we compare our findings with previous analyses of the same neuronal dataset, especially \cite{dion2019mixedsde} and \cite{delattre2021review}. Since the models considered in these studies are different from those investigated here, a direct comparison of the estimated parameter values is not meaningful. Indeed, previous works mainly relied on Ornstein--Uhlenbeck mixed-effects models.}

\tcr{Nevertheless, the main qualitative features of the data are consistent with previous analyses. In particular, our results indicate substantial between-trajectory variability, mainly captured through the random effects in the drift term. Conversely, the estimated variability associated with the random effect in the diffusion coefficient remains relatively limited. This observation is in agreement with previous studies of the same dataset, where the diffusion coefficient was assumed constant \cite{dion2019mixedsde} or where only limited heterogeneity was attributed to this component \cite{delattre2021review}.}


\section{Discussion}
\label{sec:discussion}

The stochastic hierarchical model we consider, based on a diffusion process with both fixed and random effects in the drift and the diffusion coefficient, is well-suited for analyzing repeated temporal dynamics across individuals. It provides a flexible and interpretable framework to disentangle the intrinsic variability of the process—captured through the stochastic component of the diffusion—and the inter-individual variability, modeled by treating certain parameters as random variables. Such a modeling strategy is particularly relevant for data observed in discrete time but governed by an underlying continuous-time dynamic, which is common in biomedical settings such as pharmacokinetics, neuroscience, or longitudinal physiological monitoring.
One of the strengths of our approach lies in the simplicity of the estimation strategy, which is easy to implement and computationally efficient. The procedure is broadly applicable, as it accommodates a wide range of distributions for the random effects in the diffusion coefficient, as long as they are supported on the positive real line. Furthermore, we establish the consistency and asymptotic normality of the resulting estimators, which ensures sound statistical inference. A particularly appealing feature of the method is its capacity to characterize the observed variability in the data through the estimated parameters, providing insight into both individual-level dynamics and group-level heterogeneity. 
It is worth noting that our framework also allows for extension by incorporating an auxiliary process as a covariate into the coefficients, which can significantly refine the modeling of inter-individual variability and improve explanatory power in applications.

However, the proposed methodology also presents some limitations. First, it relies on high-frequency observations, which may not be available in all practical scenarios. Second, the current model does not explicitly account for measurement noise, an aspect that can be critical in many experimental or clinical data settings. One possible extension to address low-frequency sampling could involve simulating additional intermediate points via Brownian bridges \cite{bladt2016simulation}, though this remains a partial workaround and does not fully compensate for the absence of an explicit noise model. The Python codes for this work are hosted on GitHub (\url{https://github.com/madelattre/Paper----QL-MSDE-HF----Python-codes}).

\appendix


\section{Proofs}
\label{hm:sec_proofs}

\subsection{Generic notation}
\label{hm:sec_proof.notation}

For convenience, we will use the following notation and convention throughout this section.

\begin{itemize}
\item We write
\begin{equation}
    \sup_i = \sup_{N\ge 1}\max_{i\le N}, \qquad \sup_j =\sup_{n\ge 1}\max_{j\le n}, \qquad \sup_{i,j}=\sup_{n,N\ge 1}\max_{i\le N, \, j\le n},
\end{equation}
where the supremum $\sup_{n,N\ge 1}$ 
is taken under \eqref{hm:sampling.design}.

    \item For any positive sequence $(r_{N,n})$ and random functions $\xi_{N,n}(\vt)$, we write:
\begin{equation}
 \xi_{N,n}(\vt)=\left\{
\begin{array}{ll}
    O^\ast_{p}(r_{N,n}) & \text{if}\quad 
\sup_\vt \left|r_{N,n}^{-1}\xi_{N,n}(\vt)\right|=O_p(1); \\[2mm]
    o^\ast_{p}(r_{N,n}) & \text{if}\quad 
 \sup_\vt \left|r_{N,n}^{-1}\xi_{N,n}(\vt)\right| = o_p(1).
\end{array}
 \right.
 \nn
\end{equation}

    \item We write $\overline{\mfm}_{ij}$ for any random variable (possibly matrix-valued) which is $\mcf_{i,t_j}$-measurable 
and satisfies that
\begin{equation}\nn
    \forall K>0\quad \sup_{i,j}\E\big[|\overline{\mfm}_{ij}|^K\big]<\infty.
\end{equation}
Similarly, the notation $\overline{\mfm}_{i}$ is used for any 
$\mcf_{i,T}$-measurable 
random variable such that
\begin{equation}\nn
    \forall K>0\quad \sup_{i}\E\big[|\overline{\mfm}_{i}|^K\big]<\infty.
\end{equation}
Obviously, it is true that $\overline{\mfm}_{i}=n^{-1}\sum_j \overline{\mfm}_{ij}$ and $\overline{\mfm}_{i} = \overline{\mfm}_{ij}$.
Further, we generically write $\overline{\mfb}_{ij}$ for any essentially bounded $\mcf$-measurable random variable possibly depending on the index $(i,j)$:
\begin{equation}
    \max_{i,j}|\overline{\mfb}_{ij}| < \infty, \qquad \text{a.s.}
\end{equation}
The above variables may vary at each appearance, and so are their dimensionalities.

\end{itemize}

\subsection{Proof of Theorem \ref{hm:thm_diff.param}}
\label{hm:sec_proof.diff}

This section is devoted to the proof of Theorem \ref{hm:thm_diff.param}. 
We will handle the fixed-effect and random-effect parameters step by step in this order, which is possible due to the mixed-rates structure of the first-stage profile quasi-likelihood $\mbbh_1(\vt_1)$.

\subsubsection{Fixed-effect parameter}
\label{hm:sec_diff.fixed.param}

\paragraph{Step 1: Consistency.}~

We begin with proving the consistency of $\wh{\eta}$.
Introduce the random function
\begin{equation}
\mbby_{11}(\eta)=\mbby_{11,N,n}(\eta) := \frac{1}{nN}\left( \mbbh_{11}(\eta) - \mbbh_{11}(\eta_{0}) \right),
\nonumber
\end{equation}
which is maximized at $\wh{\eta}$.
By the standard $M$-estimation theory,
the consistency of $\wh{\eta}$ will follow from the uniform-in-$\eta$ convergence in probability of $\mbby_{11}$ and the suitable identifiability condition.

We will use the following notation:
\begin{align}
    A_{11}(y,\eta) &:= \log\left(\frac{S(y;\eta)}{S(y;\eta_{0})}\right),
    \nn\\
    \overline{A}_{11,T}(\eta)
    &:=\E\left[\frac{1}{T}\int_0^T A_{11}(Y_1(t);\eta)dt\right],
    \nn\\
    B_{11}(x,\eta) &:= \exp\left(-A_{11}(y;\eta)\right)
    =\frac{S(y;\eta_{0})}{S(y;\eta)},
    \nn\\
    \overline{B}_{11,T}(\eta)
    &:=\E\left[\log\left(\frac{1}{T}\int_0^T B_{11}(Y_1(t);\eta)dt\right)\right].
    \nn
\end{align}
Note that $\overline{A}_{11,T}(\eta_{0})=\overline{B}_{11,T}(\eta_{0})=0$.
Let
\begin{align}
    \mbby_{11,0}(\eta) &:= -\frac12\left( \overline{A}_{11,T}(\eta) + \overline{B}_{11,T}(\eta) \right)
    \nn\\
    &= \frac12
    \E\bigg[
    -\log\bigg(\frac{1}{T}\int_0^T 
    \frac{S_1(t)}{S_1(t;\eta)}dt\bigg)
    +\frac{1}{T}\int_0^T 
    \log\bigg(\frac{S_1(t)}{S_1(t;\eta)}\bigg)dt
    \bigg].
    \nn
\end{align}
Because of the strict convexity of $s\mapsto -\log s$ ($s>0$) and Assumption \ref{hm:A_iden}, Jensen's inequality ensures that the random variables inside the last expectation are a.s. negative for $\eta\ne\eta_{0}$. This implies that $\mbby_{11,0}(\eta)\le 0$ with the equality holding only for $\eta=\eta_{0}$.
Hence, the consistency of $\wh{\eta}$ follows from showing that
\begin{equation}\label{hm:consis-5}
    \mbby_{11}(\eta) - \mbby_{11,0}(\eta) = o_p^\ast(1).
\end{equation}

To deduce \eqref{hm:consis-5}, we will look at the two terms on the right-hand side of \eqref{hm:def_QLF1-1} separately.
To proceed, we prove the following auxiliary lemma.

\begin{lem}
\label{hm:aux.lem1}
Let $\psi(y):\,\mbbr^q\times \mbbr\to\mbbr$ be a $\mcc^1$-class function such that
\begin{equation}
\left(|\psi(y)| + |\p_{y} \psi(y)| \right)\lesssim 1+|y|^C.
\nonumber
\end{equation}
Then, we have for any $K\ge 2$,
\begin{align}
    & \sup_i \sup_n \E\Bigg[\Bigg|\sqrt{n}\Bigg(
    \frac{1}{n}\sum_{j}\psi_{i,j-1} - \frac{1}{T}\int_0^T \psi_i(t)dt
    \Bigg)\Bigg|^K\Bigg] <\infty,
    \label{hm:aux.lem1-0}\\
    & \sup_{N,n}\E\Bigg[\Bigg|\sqrt{N}\Bigg(
    \frac{1}{nN}\sum_{i,j}\psi_{i,j-1} - \E\Bigg[\frac{1}{T}\int_0^T \psi_1(t)dt\Bigg]
    \Bigg)\Bigg|^K\Bigg] <\infty,
    \label{hm:aux.lem1-1}\\
    & \sup_{N,n}\E\Bigg[\Bigg|\sqrt{N}\Bigg\{
    \frac{1}{N}\sum_{i} \Bigg(\frac1n \sum_j \psi_{i,j-1}\Bigg)^{\otimes 2} 
    \nn\\
    &{}\qquad \qquad 
    -\E\Bigg[\Bigg(\frac{1}{T}\int_0^T \psi_1(t)dt\Bigg)^{\otimes 2}\Bigg]
    \Bigg\}\Bigg|^K\Bigg] <\infty.
    \label{hm:aux.lem1-2'}
\end{align}
\end{lem}

\begin{proof}
We only prove \eqref{hm:aux.lem1-1}, the proofs of \eqref{hm:aux.lem1-0} and \eqref{hm:aux.lem1-2'} are similar.
Recalling that $h=T/n$ and that $Y_1,Y_2,\dots$ are i.i.d., we can estimate as
\begin{align}
    & \left|\frac{1}{nN}\sum_{i,j}\psi_{i,j-1} -\frac{1}{N}\sum_i\frac{1}{T}\int_0^T \E[\psi_1(t)]dt\right|
    \nn\\
    &\le \frac{1}{N} \sum_i \frac1n \sum_j \frac1h \int_{j}\left|\psi_i(t) - \psi_{i,j-1} \right| dt
    \nn\\
    &{}\qquad
    + \left|\frac{1}{N} \sum_i \frac1n \sum_j \frac1h \int_{j}\left(\psi_i(t) - \E[\psi_1(t)]\right) dt\right|.
    \label{hm:aux.lem1-2}
\end{align}
The right-hand side of \eqref{hm:aux.lem1-2} is bounded by $\max\{\sqrt{h},N^{-1/2}\}\lesssim N^{-1/2}$ multiplied by the sum
\begin{align}
& \frac{1}{N} \sum_i \frac1n \sum_j \frac1h \int_{j}
    \int_0^1 (1+|Y_i(t'_{i,j}(s))|^C)ds\,
    \left|\frac{1}{\sqrt{h}}\left(Y_i(t) - Y_i(t_{i,j-1})\right) \right| dt
    \nn\\
    &{}\qquad
    + \left|\frac{1}{\sqrt{N}} \sum_i \frac1n \sum_j \frac1h \int_{j}\left(\psi_i(t) - \E[\psi_1(t)]\right) dt\right|,
    \label{hm:aux.lem1-3}
\end{align}
where $t'_{i,j}(s):=t_{i,j-1}+s(t-t_{i,j-1})$. 
By
\eqref{hm:y.ij_mb} and the Burkholder inequality, \eqref{hm:aux.lem1-3} is $L^K(\pr)$-bounded for any $K\ge 2$, giving \eqref{hm:aux.lem1-1}.
\end{proof}

We will use the following form of the Sobolev inequality (see, for example, \cite[Theorem 1.4.2]{Ada73}):
given a bounded convex domain $\mathsf{D}\subset\mbbr^{s}$ for some $s\ge 1$ and for any $\mcc^1(\mathsf{D})$-function $\psi$ and a constant $\al > s$, we have
\begin{equation}
\label{hm:sobolev.ineq}
    \sup_{y\in\mathsf{D}} |\psi(y)|^\al \le C \int_{\mathsf{D}} \left(|\psi(y)|^\al +|\p_{y} \psi(y)|^\al\right) dy
\end{equation}
for some universal constant $C=C(\mathsf{D}, \al) >0$.

Recall the shorthand $S_{i,j-1}=S_{i,j-1}(\eta_{0})$ and pick any $K\ge 2$. 
Applying Lemma \ref{hm:aux.lem1}, we readily obtain for each $\eta$ and $k=0,1$,
\begin{align}
& \sup_{\theta_1}\sup_{n,N}\E\Bigg[\Bigg|\sqrt{N}\Bigg(\frac{1}{nN} \sum_{i,j} \p_{\eta}^k\log\left(\frac{S_{i,j-1}(\eta)}{S_{i,j-1}}\right)
\nn\\
&{}\qquad 
- \frac{1}{T}\int_0^T \E\left[\p_{\eta}^k A_{11}(Y_1(t);\eta)\right] dt\Bigg) 
\Bigg|^K\Bigg] = O(1).
\nn
\end{align}
These estimates together with \eqref{hm:sobolev.ineq} yield
\begin{align}
& \sqrt{N}\Bigg(\frac{1}{nN} \sum_{i,j} \log\left(\frac{S_{i,j-1}(\eta)}{S_{i,j-1}}\right)
\nn\\
&{}\qquad\qquad - \frac1N \sum_i \frac{1}{T}\int_0^T \E\left[A_{11}(Y_i(t);\eta)\right]dt\Bigg) = O^\ast_{p}(1).
\nn
\end{align}
Then, by the definition of $\mbby_{11}(\eta)$,
\begin{align}
    & \mbby_{11}(\eta) + \frac12 \overline{A}_{11,T}(\eta)   \nn\\
    &= -\frac{1}{2N} \sum_i \left\{
    \log\left(Q_i(\eta)\right) - \log\left(Q_i(\eta_{0})\right)\right\}+ O^\ast_{p}(N^{-1/2}),
    \label{hm:consis-1}
\end{align}
where
\begin{equation}\label{hm:def_Qi}
    Q_i(\eta) := \frac1n \sum_j S_{i,j-1}^{-1}(\eta)\,\ny_{ij}^2.
\end{equation}

We will identify the uniform-in-$\eta$ limit in probability of the first term on the right-hand side of \eqref{hm:consis-1}. Let us write
\begin{equation}\nn
    \ny_{ij}
= \sqrt{\tau_i} \, c_{i,j-1}(\eta_{0}) z_{ij} + \sqrt{h} \,\mu_{ij},
\end{equation}
where $z_{ij}:=h^{-1/2}\D_j w_i\sim \text{i.i.d.}~N(0,1)$ (with respect to both $i$ and $j$) and where
\begin{equation}\label{hm:mu.ij_def}
    \mu_{ij} := \mu_{ij}^\star + \sqrt{h}\,\overline{\mu}_{ij}
\end{equation}
with
\begin{align}
    \mu_{ij}^\star &:= \tau_i\,\vp_i\cdot a_{i,j-1}
    \nn\\
    &{}\qquad + \sqrt{\tau_i} \, \frac{1}{\sqrt{h}}\int_j \frac{1}{\sqrt{h}} 
    \left(c_i(s) - c_{i,j-1}(\eta_{0})\right)dw_i(s),
    \label{hm:mu.star.ij_def}\\
    \overline{\mu}_{ij} &:= \frac1h \tau_i \,\vp_i \cdot \int_j \frac{1}{\sqrt{h}}(a_i(s)-a_{i,j-1})ds.
    \nn
\end{align}
We have $\mu_{ij}^\star=\tau_i\,\vp_i\cdot a_{i,j-1}\overline{\mfm}_{ij}=\overline{\mfm}_{ij}$ and $\overline{\mu}_{ij}=\overline{\mfm}_{ij}$ (hence $\mu_{ij}=\overline{\mfm}_{ij}$) by applying \eqref{hm:tau-moments} and the Burkholder inequality under the present assumptions.
It follows that
\begin{align}\label{hm:yij^2_se}
    \ny_{ij}^2 
    &= \tau_i \,S_{i,j-1} z_{ij}^2+ \sqrt{h}\, \overline{\mfm}_{ij} \nn\\
    &=\tau_i \,S_{i,j-1} + \tau_i \,S_{i,j-1} (z_{ij}^2 -1)
    + \sqrt{h}\, \overline{\mfm}_{ij}.
\end{align}
The finer representation \eqref{hm:mu.ij_def} is not used here, but will matter later in the proof of the asymptotic normality.

By \eqref{hm:yij^2_se},
\begin{align}
    Q_i(\eta)
    = \tau_i F_i(\eta) + \sqrt{h}\left\{\del_{1,i}(\eta)+\del_{2,i}(\eta)+\del_{3,i}(\eta)\right\},
    \label{hm:consis-2}
\end{align}
where
\begin{align}
F_i(\eta) &:= \frac{1}{T} \int_0^T B_{11}(Y_i(t);\eta)dt, \nn\\
\del_{1,i}(\eta) &:= \frac{\tau_i}{\sqrt{h}} \bigg(
    \frac{1}{n} \sum_j B_{11,i,j-1}(\eta)
    - \frac1T \int_0^T B_{11}(Y_i(t);\eta)dt \bigg),
\nn\\
\del_{2,i}(\eta) &:= \frac{\tau_i}{n\sqrt{h}} \sum_j B_{11,i,j-1}(\eta) (z_{ij}^2 -1),
\nn\\
\del_{3,i}(\eta) &:= \frac{1}{n} \sum_j S_{i,j-1}^{-1}(\eta) \overline{\mfm}_{ij};
\nn
\end{align}
the last one does not really define $\del_{3,i}(\eta)$ since $\overline{\mfm}_{ij}$ is a generic notation, but that form will be enough.
By Assumption \ref{hm:A_regul}, the random function $F_i(\eta)$ is essentially bounded and also there exists a constant $K>0$ for which
\begin{equation}\label{hm:F_l.b.}
    F_i(\eta) \ge K \qquad\text{a.s.}
\end{equation}
Pick any $M>2\vee p_{\eta}$. By a similar argument to \eqref{hm:aux.lem1-3} in the proof of Lemma \ref{hm:aux.lem1} together with the Sobolev inequality \eqref{hm:sobolev.ineq}, we obtain $\sup_{i,j} \E[\sup_{\eta}|\del_{1,i}(\eta)|^M] < \infty$.
Similarly, since $n^{-1}h^{-1/2}=T^{-1/2}n^{-1/2}$, the Sobolev and Burkholder inequalities lead to $\sup_{i,j} \E[\sup_{\eta}|\del_{2,i}(\eta)|^M] < \infty$. The bound $\sup_{i,j} \E[\sup_{\eta}|\del_{3,i}(\eta)|^M] < \infty$ is trivial.
It follows that
\begin{equation}\nn
    \forall M'>0\quad
    \sup_{i,j} \E\left[\sup_{\eta}\left|\del_{i}(\eta)\right|^{M'}\right] < \infty,
\end{equation}
where $\del_{i}(\eta):=\del_{1,i}(\eta)+\del_{2,i}(\eta)+\del_{3,i}(\eta)$.

Recalling the expression \eqref{hm:consis-1}, we now take the logarithm of $Q_i(\eta)$:
by \eqref{hm:consis-2}, we have for each $\eta$,
\begin{align}
& \frac1N \sumi \log\left(Q_i(\eta)\right) - \overline{B}_{11,T}(\eta)
\nn\\
&{}\qquad = \frac1N \sumi \log\tau_i 
+ \frac1N \sumi \left(\log F_i(\eta)- \overline{B}_{11,T}(\eta)\right)
\nn\\
&{}\qquad\qquad  + \frac1N \sumi \log\left(1 + \frac{\sqrt{h}\,\del_i(\eta)}{\tau_i F_i(\eta)}\right).
\label{hm:consis-3}
\end{align}
By the argument based on the Sobolev and Burkholder inequalities as before, the second term in the right-hand side of \eqref{hm:consis-3} equals $O_p^\ast(N^{-1/2})=o_p^\ast(1)$; by definition, the second term vanishes a.s. for $\eta=\eta_{0}$.
Let
\begin{equation}\nn
    D(\eta)=D_N(\eta):=
    \frac{1}{N} \sumi \log\left(1 + \frac{\sqrt{h}\,\del_i(\eta)}{\tau_i F_i(\eta)}\right).
\end{equation}
Then, by \eqref{hm:consis-1} and \eqref{hm:consis-3} we have
\begin{align}\nn
\mbby_{11}(\eta) - \mbby_{11,0}(\eta)
= -\frac{1}{2} D(\eta) + \frac{1}{2} D(\eta_{0}) + O_p^\ast(N^{-1/2}).
\end{align}
Pick any sequence $(\kappa_N)\subset(0,1/2]$ such that $\kappa_N\to 0$ and that $\sqrt{n}\,\kappa_N \gtrsim N^{\mathsf{a}_{11}}$ for some constant $\mathsf{a}_{11}>0$ (possible because of \eqref{hm:sampling.design}). We will show that
\begin{equation}\label{hm:D11_nble}
    \forall\mathsf{b}\in(0,1)\quad D(\eta)= o_p^\ast(\kappa_N^{\mathsf{b}}),
\end{equation}
which concludes the following estimate slightly stronger than \eqref{hm:consis-5}:
\begin{equation}\label{hm:consis-4}
    \mbby_{11}(\eta) - \mbby_{11,0}(\eta) = O_p^\ast\left(\max\{N^{-1/2},\kappa_N^{\mathsf{b}}\}\right)=o_p^\ast(1).
\end{equation}

To verify \eqref{hm:D11_nble}, we introduce the event
\begin{equation}\nn
    G_N:=\left\{ 
    \max_{i\le N}\sup_{\eta}\left|\frac{\sqrt{h}\,\del_i(\eta)}{\tau_i F_i(\eta)}\right|>\kappa_N
    \right\},
\end{equation}
and fix any $\ep>0$ and $\mathsf{b}\in(0,1)$. Then,
\begin{align}\nn
    \pr\left[\sup_{\eta}|D(\eta)|>\kappa_N^{\mathsf{b}}\,\ep\right]
    \le \pr[G_N] + \pr\left[G_N^c \cap \left\{\sup_{\eta}|D(\eta)|>\kappa_N^{\mathsf{b}}\,\ep\right\}\right].
\end{align}
For any $M>\mathsf{a}_{11}^{-1}$, recalling \eqref{hm:F_l.b.} and applying the Markov and Cauchy-Schwarz inequalities, we have for $M>0$,
\begin{align}
    \pr[G_N]
    &\le \sumi \pr\left[\tau_i^{-1}\sup_{\eta}|\del_i(\eta)| \ge C \frac{\kappa_N}{\sqrt{h}}\right]
    \nn\\
    &\lesssim \sumi 
    \left(\frac{\sqrt{h}}{\kappa_N}\right)^M
    \E\left[\tau_i^{-M}\sup_{\eta}|\del_i(\eta)|^M\right]
    \nn\\
    &\lesssim \frac{N}{(\sqrt{n}\,\kappa_N)^M} \lesssim N^{1-\mathsf{a}_{11}M} \to 0.
    \label{hm:consis-6}
\end{align}
By the elementary inequality $|\log(1+x)|\le C |x|$ for $|x|\le 1/2$, we have
\begin{align}
    \pr\left[G_N^c \cap \left\{\sup_{\eta}|D(\eta)|>\kappa_N^{\mathsf{b}}\,\ep\right\}\right]
    &\le \pr\left[C \kappa_N^{1-\mathsf{b}} > \ep\right] = 0
    \nn
\end{align}
for every $N$ large enough. We conclude \eqref{hm:D11_nble}, hence the consistency $\wh{\eta} \cip \eta_{0}$.

\medskip

\paragraph{Step 2: Asymptotic normality.}~

Now we turn to the proof of the asymptotic normality of $\wh{\eta}$.
First, we study the asymptotic behavior of the \textit{quasi-score function}
\begin{equation}\nn
    \D_{11} := \frac{1}{\sqrt{nN}} \p_{\eta}\mbbh_{11}(\eta_{0}) = \frac{1}{\sqrt{nN}} \p_{\eta}\mbbh_{11}.
\end{equation}
Recalling the notation \eqref{hm:def_Qi} and \eqref{hm:def_g}, we can write
\begin{align}\label{hm:s11-1}
\D_{11} = \frac{1}{2\sqrt{N}} \sumi \left(\tau_i^{-1}Q_i\right)^{-1}\frac{1}{\sqrt{n}}\sum_j \left(\frac{\ny_{ij}^2}{\tau_i S_{i,j-1}} - \tau_i^{-1} Q_i\right)\, g_{i,j-1}.
\end{align}
We need to have a closer look at the right-hand side to extract the leading term for which we can apply the central limit theorem.

\medskip

We have $F_i(=F_i(\eta_{0}))=1$, $\del_{1,i}=0$, and $\tau_i^{-1}\del_{2,i}=(nT)^{-1/2}\sum_j (z_{ij}^2 -1)$ in the expression \eqref{hm:consis-2}.
Hence
\begin{align}
    \tau_i^{-1}Q_i
    &= 1+\sqrt{h} \,\frac{\del_i}{\tau_i}
    = 1 + \sqrt{h} \, \left(\frac{1}{\sqrt{T n}}\sum_j (z_{ij}^2 -1) + \tau_i^{-1}\, \del_{3,i} \right)
    \label{hm:tauinv.Q}
\end{align}
and obviously $\tau_i^{-1}Q_i=1+ \sqrt{h}\, \overline{\mfm}_{i}$.
By \eqref{hm:yij^2_se},
\begin{align}\label{hm:s11-3}
    \frac{\ny_{ij}^2}{\tau_i S_{i,j-1}} - \tau_i^{-1} Q_i 
    &= \zeta_{ij} + \sqrt{h}\,\overline{\mfm}_{ij},
\end{align}
where
\begin{equation}\label{hm:zeta_def}
    \zeta_{ij} := (z_{ij}^2 -1) - \frac1n \sum_{k=1}^{n} (z_{ik}^2 -1)
    =(z_{ij}^2-1) + \frac{1}{\sqrt{n}}\overline{\mfm}_i.
\end{equation}
Substituting \eqref{hm:tauinv.Q} and \eqref{hm:s11-3} into \eqref{hm:s11-1} and rearranging the resulting terms, we have
\begin{align}
    \D_{11} &= \frac{1}{2\sqrt{N}} \sumi \left(1+\sqrt{h} \,\frac{\del_i}{\tau_i}\right)^{-1}\frac{1}{\sqrt{n}}\sum_j \left(\zeta_{ij} + \sqrt{h}\,\overline{\mfm}_{ij}\right)\,g_{i,j-1}
    \nn\\
    &= \D_{11}^\star + \D_{11}^{r,1}+\D_{11}^{r,2}+\D_{11}^{r,3},
\end{align}
where
\begin{align}
    \D_{11}^\star &:= \frac{1}{2\sqrt{N}} \sumi \frac{1}{\sqrt{n}}\sum_j g_{i,j-1}\zeta_{ij},
    \nn\\
    \D_{11}^{r,1} &:= -\frac{1}{2\sqrt{N}} \sumi \sqrt{h} \,\frac{\del_i}{\tau_i} \frac{1}{\sqrt{n}}\sum_j \left(\zeta_{ij} + \sqrt{h}\,\overline{\mfm}_{ij}\right)\,g_{i,j-1},
    \nn\\
    \D_{11}^{r,2} &:= \frac{1}{2\sqrt{N}} \sumi \left\{\left(1+\sqrt{h} \,\frac{\del_i}{\tau_i}\right)^{-1}\!\!-\left(1-\sqrt{h} \,\frac{\del_i}{\tau_i}\right)\right\}
    \label{hm:Dr2_def}\\
    &{}\qquad \times
    \frac{1}{\sqrt{n}}\sum_j \left(\zeta_{ij} + \sqrt{h}\,\overline{\mfm}_{ij}\right)\,g_{i,j-1},
    \nn\\
    \D_{11}^{r,3} &:= \frac{1}{2\sqrt{N}} \sumi \frac{1}{\sqrt{n}}\sum_j \sqrt{h}\,\overline{\mfm}_{ij} \,g_{i,j-1}.
\end{align}
We will separately show that $\D_{11}^{r,k}=o_p(1)$ for $k=1,2,3$.

By \eqref{hm:zeta_def}, we have
\begin{align}
    \frac{1}{\sqrt{n}}\sum_j g_{i,j-1}\zeta_{ij} 
    = \frac{1}{\sqrt{n}}\sum_j g_{i,j-1}(z_{ij}^2-1) 
    + \frac{1}{n}\sum_j g_{i,j-1} \overline{\mfm}_i = \overline{\mfm}_i.
\end{align}
Therefore, it is easily seen that
\begin{align}
    |\D_{11}^{r,1}| &\lesssim 
    \sqrt{\frac{N}{n}}\, \frac1N \sumi |\overline{\mfm}_i|
    \left(
    \frac{1}{\sqrt{n}}\sum_j g_{i,j-1}\zeta_{ij} + 
    \frac{1}{n}\sum_j g_{i,j-1}\overline{\mfm}_{ij}
    \right)
    \nn\\
    &\lesssim \sqrt{\frac{N}{n}}\, \frac1N \sumi |\overline{\mfm}_i|
    =O_p\left(\sqrt{\frac{N}{n}}\right) = o_p(1),
\end{align}
concluding that $\D_{11}^{r,1}=o_p(1)$.

To deal with $\D_{11}^{r,2}$, we note that 
\begin{equation}\label{hm:max.hc.mi}
    \forall\mathsf{c}>0,\quad \max_{i\le N}\big|h^{\mathsf{c}}\,\overline{\mfm}_{i} \big|\cip 0.
\end{equation}
Indeed, in a similar manner to \eqref{hm:consis-6}, we have for $\ep>0$,
\begin{equation}
    \pr\left[ \max_{i\le N}| h^{\mathsf{c}}\,\overline{\mfm}_{i}| > \ep \right]
    \le \sumi \pr\left[|\overline{\mfm}_i|>\frac{\ep}{h^{\mathsf{c}}}\right]
    \lesssim Nh^{K \mathsf{c}} \lesssim n^{\mathsf{a}''-K \mathsf{c}}\to 0
\end{equation}
by taking $K>0$ sufficiently large. 
In particular, \eqref{hm:max.hc.mi} ensures that 
\begin{equation}
    \max_{i\le N}|\sqrt{h} \,\tau_i^{-1}\del_i|=o_p(1).
\end{equation}
Further, we note the elementary fact
\begin{equation}
\sup_{|x|\le 1/2} x^{-2}\left| (1+x)^{-1} - (1 - x) \right| <\infty.
\end{equation}
Based on these observations, for any $\ep>0$,
\begin{align}
    \pr[|\D_{11}^{r,2}|>\ep]
    &\le \pr\left[ \max_{i\le N}\left|\sqrt{h} \,\frac{\del_i}{\tau_i}\right|> \frac12
    \right] + \pr\left[ |\D_{11}^{r,2}|>\ep,~\max_{i\le N}\left|\sqrt{h} \,\frac{\del_i}{\tau_i}\right|\le \frac12\right]
    \nn\\
    &\le o(1) + \pr\left[ \frac1N \sumi |\overline{\mfm}_i| \ge C\ep \frac{n}{\sqrt{N}}\right]
    \lesssim o(1) + \frac{\sqrt{N}}{n} = o(1).
\end{align}
Hence $\D_{11}^{r,2}=o_p(1)$.

Turning to $\D_{11}^{r,3}$, we need to look at its leading terms in a more detailed manner.
We trace back the manipulations \eqref{hm:mu.ij_def}, \eqref{hm:yij^2_se}, and \eqref{hm:s11-3}. By straightforward yet a bit messy computations, it can be seen that 
\begin{align}
 \text{(The term $\overline{\mfm}_{ij}$ in \eqref{hm:s11-3})}
 &= \frac{2 \mu_{ij}z_{ij}}{\tau_i c_{i,j-1}}
 + \frac{2}{\sqrt{\tau_i}}\frac1n \sum_j \frac{\mu_{ij}^\star z_{ij}}{c_{i,j-1}}
 + \sqrt{h}\,\overline{\mfm}_{ij}.
 \end{align}
Substituting this expression into \eqref{hm:s11-3}, we can estimate $\D_{11}^{r,3}$ as follows:
\begin{align}
    |\D_{11}^{r,3}| &\lesssim 
    \frac1N \sumi \frac1n \sum_j |\overline{\mfm}_{ij}| \sqrt{\frac{N}{n}}
    + \frac1N \sumi \tau_i^{-1} \left|\frac1n \sum_j \frac{g_{i,j-1}}{c_{i,j-1}} \mu_{ij}z_{ij}\right| \sqrt{N}
    \nn\\
    &{}\qquad + \frac1N \sumi \tau_i^{-1/2} \left|\frac1n \sum_{k=1}^{n} g_{i,k-1}
    \frac1n \sum_j c_{i,j-1}^{-1}\mu_{ij}^\star z_{ij}\right| \sqrt{N}
    \nn\\
    &\lesssim O_p\left(\sqrt{\frac{N}{n}}\right)
    + \frac1N \sumi \tau_i^{-1} \left|\frac1n \sum_j \frac{g_{i,j-1}}{c_{i,j-1}} \mu_{ij}z_{ij}\right| \sqrt{N} \nn\\
    &{}\qquad + \frac1N \sumi \tau_i^{-1/2} |\overline{\mfm}_i| \left|
    \frac1n \sum_j c_{i,j-1}^{-1}\mu_{ij}^\star z_{ij}\right| \sqrt{N}
    \nn\\
    &=: O_p\left(\sqrt{\frac{N}{n}}\right) + \overline{\D}_{11}^{r,3,1} + \overline{\D}_{11}^{r,3,2}.
\end{align}
To conclude $\D_{11}^{r,3}=o_p(1)$, we will show that $\overline{\D}_{11}^{r,3,1}$ and $\overline{\D}_{11}^{r,3,2}$ are both $o_p(1)$.

For convenience, we will use the generic notation $\pi_{i,j-1}$ for any $\mcf_{i,t_{j-1}}$-measurable random variables such that
\begin{equation}
    |\pi_{i,j-1}| \le C (1+|Y_{i,j-1}|^C).
\end{equation}
Further, we will write $E^{i,j-1}[\cdot]$ for the conditional expectation $\E[\cdot|\mcf_{i,t_{j-1}}]$.

For $\overline{\D}_{11}^{r,3,1}$, by compensating the summands $g_{i,j-1} c_{i,j-1}^{-1} \mu_{ij}z_{ij}$ and applying the Burkholder inequality, we have
\begin{align}\label{hm:Dr.3-1}
    \overline{\D}_{11}^{r,3,1} &\lesssim O_p\left(\sqrt{\frac{N}{n}}\right) 
    + \frac1N \sumi \tau_i^{-1} \left|\frac1n \sum_j \pi_{i,j-1}\E^{i,j-1}[\mu_{ij}z_{ij}]\right| \sqrt{N}
\end{align}
Now we recall the expression \eqref{hm:mu.ij_def} of $\mu_{ij}$.
By the basic properties of the quadratic characteristic:
\begin{align}
    & \E^{i,j-1}\left[\int_{j}\al_i(s)ds\, \int_{j}\beta_i(s)dw_i(s)\right] 
    = 0,
    \nn\\
    & \E^{i,j-1}\left[\int_{j}\al_i(s)dw_i(s)\, \int_{j}\beta_i(s)dw_i(s)\right] 
    = \E^{i,j-1}\left[\int_{j}\al_i(s)\beta_i(s)ds\right]
\end{align}
for sufficiently integrable $(\mcf_{i,t})_{t\le T}$-adapted process $(\al_i,\beta_i)$ and also applying It\^{o}'s formula, we can deduce that the second term in the right-hand side of \eqref{hm:Dr.3-1} equals
\begin{align}
    & \frac1N \sumi \tau_i^{-1/2} \left|\frac1n \sum_j \pi_{i,j-1}
    \frac1h \int_{j}\frac{1}{\sqrt{h}}
    \E^{i,j-1}[c(Y_i(s);\eta_{0}) - c_{i,j-1}]ds\right| \sqrt{N}
    \nn\\
    &\lesssim \frac1N \sumi \tau_i^{-1/2} 
    \left|\frac1n \sum_j \pi_{i,j-1}
    \frac1h \int_{j}\frac{1}{\sqrt{h}} \,h\, \pi_{i,j-1}
    ds\right| \sqrt{N} \nn\\
    &\lesssim \frac1N \sumi |\overline{\mfm}_i|\,\sqrt{\frac{N}{n}} = O_p\left(\sqrt{\frac{N}{n}}\right).
\end{align}
We conclude that $\overline{\D}_{11}^{r,3,1}=O_p(\sqrt{N/n})=o_p(1)$.

The term $\overline{\D}_{11}^{r,3,2}$ can be handled similarly: through the compensation of $c_{i,j-1}^{-1}\mu_{ij}^\star z_{ij}$ as before, we obtain
\begin{align}
    \overline{\D}_{11}^{r,3,2}
    &\le O_p\left(\sqrt{\frac{N}{n}}\right) + 
    \frac1N \sumi |\overline{\mfm}_i| \left|
    \frac1n \sum_j c_{i,j-1}^{-1} \E^{i,j-1}[\mu_{ij}^\star z_{ij}] \right| \sqrt{N}.
\end{align}
Plugging-in the expression \eqref{hm:mu.star.ij_def} of $\mu_{ij}^\star$ and then handling it as in the case of \eqref{hm:Dr.3-1}, we get $\overline{\D}_{11}^{r,3,2}=O_p(\sqrt{N/n})=o_p(1)$.

Thus we have obtained $\D_{11}=\D_{11}^\star+o_p(1)$, hence it remains to study $\D_{11}^\star$.
Let
\begin{equation}
    \overline{g}_i := \frac1T \int_0^T g(Y_i(t))dt.
\end{equation}
Then, by Lemma \ref{hm:aux.lem1},
\begin{equation}
    \sqrt{n}\bigg(\frac1n \sum_j g_{i,j-1} - \overline{g}_i\bigg)=\overline{\mfm}_i.
\end{equation}
Hence
\begin{align}
    \D_{11}^\star 
    &= \frac{1}{2\sqrt{N}} \sumi \Bigg\{
    \frac{1}{\sqrt{n}}\sum_j g_{i,j-1}(z_{ij}^2 -1)
    \nn\\
    &{}\qquad 
    -\bigg(\frac1n \sum_{k=1}^{n} g_{i,k-1}\bigg)
    \frac{1}{\sqrt{n}} \sum_j (z_{ij}^2 -1)
    \Bigg\} + o_p(1) \nn\\
    &= \frac{1}{2\sqrt{N}} \sumi \Bigg(
    \frac{1}{\sqrt{n}}\sum_j g_{i,j-1}(z_{ij}^2 -1) -\overline{g}_i
    \frac{1}{\sqrt{n}} \sum_j (z_{ij}^2 -1) \Bigg) + o_p(1).
\end{align}
Denote by $\sumi \zeta_i$ the first term on the rightmost side.

Toward the central limit theorem for $\D_{11}^\star$, we will show that
\begin{equation}\label{hm:mart.rep-4}
    |\E[\zeta_i]| \lesssim \frac{1}{\sqrt{nN}}. 
\end{equation}
The conditioning argument gives
\begin{align}
    \E[\zeta_i] &= -\frac{1}{2\sqrt{nN}}\sum_j \E\big[\overline{g}_i (z_{ij}^2-1)\big]
    \nn\\
    &= -\frac{1}{2\sqrt{nN}}\sum_j \E\left[\E\big[(\overline{g}_i-\E[\overline{g}_i]) (z_{ij}^2-1) \middle| (\tau_i,\vp_i) \big]\right].
\end{align}
Conditional on $(\tau_i,\vp_i)$, the variables $\overline{g}_i$ and $z_{ij}$ are functionals of $\{Y_i(0),w_i\}$; recall that $(\tau_i,\vp_i)$ and $\{Y_i(0),w_i\}$ are independent.
It\^{o}'s formula gives the expression
\begin{equation}
    z_{ij}^2-1 = \frac{2}{h}\int_j (w_i(s)-w_{i,j-1})dw_i(s).
\end{equation}
By applying the representation theorem \cite[Theorem 1.1.3]{Nua06} conditional on $(\tau_i,\vp_i)$, there exists an $(\mcf_{i,t})$-adapted $L^2$-process $\psi_i=\psi_{i,T}$ such that
\begin{align}
    \overline{g}_i-\E[\overline{g}_i] 
    = \int_0^T \psi_i(s)dw_i(s).
\end{align}
Therefore,
\begin{align}
    \E[\zeta_i] &= -\frac{1}{2\sqrt{nN}}\sum_j 
    \frac{2}{h}\E\left[\E\left[
    \int_0^T \psi_i(s)
    \hmrev{dw_i(s)}
    \int_j (w_i(s)-w_{i,j-1})dw_i(s)
    \middle| \,(\tau_i,\vp_i) \right]\right]
    \nn\\
    &= -\frac{1}{2\sqrt{nN}}\sum_j \frac{2}{h}
    \E\left[\int_j \E\left[\psi_i(s) (w_i(s)-w_{i,j-1})
    \middle| \,(\tau_i,\vp_i) \right]ds\right].
    \label{hm:mart.rep-2}
\end{align}
We can write 
\begin{align}
    & \E\left[\psi_i(s) (w_i(s)-w_{i,j-1})\middle| \,(\tau_i,\vp_i) \right]
    \nn\\
    &= \E\left[\left(\psi_i(s)-\E^{i,j-1}[\psi_i(s)]\right) (w_i(s)-w_{i,j-1}) \middle| \,(\tau_i,\vp_i) \right].
    \label{hm:mart.rep-1}
\end{align}
By the martingale representation theorem \cite[Theorem III.4.33 a)]{JacShi03}, 
there exists an $(\mcf_{i,t})_{t\le T}$-predictable matrix-valued process $\psi'_i$ for which the right-hand side of \eqref{hm:mart.rep-1} is written as
\begin{equation}
    \E\left[\int_{t_{j-1}}^s\psi'_i(u)dw_i(u)
    \int_{t_{j-1}}^s dw_i(u) \middle| \,(\tau_i,\vp_i) \right]
    = \E\left[\int_{t_{j-1}}^s\psi'_i(u) du \middle| \,(\tau_i,\vp_i) \right].
    \label{hm:mart.rep-3}
\end{equation}
Now, substituting \eqref{hm:mart.rep-3} into \eqref{hm:mart.rep-2} gives \eqref{hm:mart.rep-4}:
\begin{align}
    |\E[\zeta_i]| &= \left|-\frac{1}{2\sqrt{nN}}\sum_j \frac{2}{h}
    \E\left[\int_j \int_{t_{j-1}}^s\psi'_i(u) duds
    \right]\right| \lesssim \frac{1}{\sqrt{nN}}\sumj \frac1h h^2 = \frac{1}{\sqrt{nN}}.
\end{align}

Now \eqref{hm:mart.rep-4} leads to
\begin{equation}
    \D_{11}^\star = \sumi (\zeta_i-\E[\zeta_i]) + o_p(1)=:\sumi \widetilde{\zeta}_i + o_p(1).
\end{equation}
It is easy to see that for $K>2$, $\sumi \E[|\widetilde{\zeta}_i|^K] \lesssim N^{1-K/2}\to 0$, so that the Lyapunov condition holds. The convergence of the covariance $Q_{11}:=\sumi \E[\widetilde{\zeta}_i^{\otimes 2}]$ remains to be studied.

Observe that $Q_{11}=\sumi \E[\zeta_i^{\otimes 2}]+O(n^{-1})$ by \eqref{hm:mart.rep-4}.
Since $(\tau_i,\vp_i,Y_i)$, $i\ge 1$, are i.i.d., we have
\begin{align}\label{hm:mart.rep-5}
    \sumi \E[\zeta_i^{\otimes 2}]
    &= \frac{1}{4}\E\left[\{ M_{11,1} - \overline{g}_1 M_{11,0}\}^{\otimes 2}\right],
\end{align}
where $M_{11,1} := n^{-1/2}\sum_j g_{i,j-1}(z_{ij}^2 -1)$ and $M_{11,0} := n^{-1/2}\sum_j(z_{ij}^2 -1)$. 
We can apply \cite[Theorem IX.7.28]{JacShi03} for the step process
\begin{equation}
    \overline{M}_{11}(t) := \sum _{j=1}^{[t/h]}\frac{1}{\sqrt{n}}
    \begin{pmatrix}
        g_{1,j-1} \\ 1
    \end{pmatrix}(z_{ij}^2-1),\qquad 0\le t\le T,
\end{equation}
to conclude that $\overline{M}_{11}(\cdot)$ converges $\mcf$-stable in law (with respect to the Skorokhod topology) to the mixed-normal limit
\begin{align}
    \overline{M}_{11,\infty}(t) 
    &=\left( M_{11,1,\infty}(t), M_{11,0,\infty}(t) \right)
    \nn\\
    &\sim \sqrt{2}
    \begin{pmatrix}
        T^{-1}\int_0^T g_1(t)^{\otimes 2}dt & \text{sym.} \\
        \overline{g}_1 & 1
    \end{pmatrix}^{1/2} \mathsf{Z}
\end{align}
defined on a suitable extended probability space, where $\mathsf{Z} \sim N_{p_{\eta_{1}}}(0,I_{p_{\eta_{1}}})$ is independent of $\mcf$; to apply the theorem, we set 
\hmrev{$w_1(\cdot)$} to be a reference continuous local martingale, and make use of the representation theorem \cite[Theorem III.4.33 a)]{JacShi03} as before to verify the asymptotic orthogonality condition \cite[IX.7.31]{JacShi03}.
Since $\overline{g}_i$ is $\mcf$-measurable, the $\mcf$-stability of the convergence $\overline{M}_{11}(\cdot) \cil \overline{M}_{11,\infty}(\cdot)$ guarantees the joint weak convergence
\begin{equation}
    (\overline{M}_{11}(T),\overline{g}_1) \cil (\overline{M}_{11,\infty}(T),\overline{g}_1), \qquad n\to\infty.
\end{equation}
Then, the continuous mapping theorem gives
\begin{equation}
    M_{11,1} - \overline{g}_i M_{11,0} \cil M_{11,1,\infty}(T) - \overline{g}_1 M_{11,0,\infty}(T),\qquad n\to\infty.
\end{equation}
The moment boundedness
\begin{equation}
    \sup_n\E\left[|(\overline{M}_{11}(T),\overline{g}_1)|^K\right] < \infty
\end{equation}
for any $K>0$ is obvious. Hence, the sequence $( \{ M_{11,1} - \overline{g}_1 M_{11,0}\}^{\otimes 2} )_n$ is uniformly integrable, which ensures the convergence of moments. By \eqref{hm:mart.rep-5},
\begin{align}
    \sumi \E[\zeta_i^{\otimes 2}]
    &\to \frac14 \E\left[
    \left\{M_{11,1,\infty}(T) - \overline{g}_1 M_{11,0,\infty}(T)\right\}^{\otimes 2}
    \right]
    \nn\\
    &= \frac12 \E\left[
    \frac1T \int_0^T g_1(t)^{\otimes 2}dt - \overline{g}_1^{\otimes 2}
    \right] =: Q_{11,0},
\end{align}
where the limit is positive definite under Assumption \ref{hm:A_iden}(2).
Thus, we have arrived at
\begin{equation}\label{hm:s11-7}
    \D_{11} = \D_{11}^\star + o_p(1) \cil N_{p_{\eta}}(0,Q_{11,0}).
\end{equation}

\medskip

Now we turn to the \textit{quasi-observed information matrix}
\begin{equation}
    \Gam_{11} := -\frac{1}{nN}\p_{\eta}^2 \mbbh_{11}(\eta_{0})= -\frac{1}{nN}\p_{\eta}^2 \mbbh_{11}.
\end{equation}
We observe that
\begin{align}
    \Gam_{11} &= \frac{1}{2N}\sumi \Bigg\{\frac1n \sum_j (\p_{\eta}^2\log S)_{i,j-1}
    + Q_i^{-1} \frac1n \sum_j (\p_{\eta}^2(S^{-1}))_{i,j-1}\ny_{ij}^2
    \nn\\
    &{}\qquad - Q_i^{-2}\bigg( \frac1n \sum_j (\p_{\eta}(S^{-1}))_{i,j-1}\ny_{ij}^2 \bigg)^{\otimes 2}\Bigg\}
    \nn\\
    &= \frac{1}{2N}\sumi \Bigg\{\frac1n \sum_j (\p_{\eta}^2\log S)_{i,j-1}
    + \frac1n \sum_j (\p_{\eta}^2(S^{-1}))_{i,j-1} S_{i,j-1}z_{ij}^2
    \nn\\
    &{}\qquad - \bigg( \frac1n \sum_j (\p_{\eta}(S^{-1}))_{i,j-1} S_{i,j-1}z_{ij}^2 \bigg)^{\otimes 2}\Bigg\} + o_p(1)
    \nn\\
    &= \frac{1}{2} \Bigg\{ \frac{1}{nN} \sumi\sum_j g_{i,j-1}^{\otimes 2} 
    - \frac1N \sumi \bigg( \frac1n \sum_j g_{i,j-1}\bigg)^{\otimes 2}\Bigg\} + o_p(1).
    \nn
\end{align}
In the second equality of the above computations, we used \eqref{hm:tauinv.Q}, \eqref{hm:yij^2_se}, and also the same argument to derive $\D_{11}^{r,2}=o_p(1)$ (see \eqref{hm:Dr2_def}) to handle the reciprocal $Q_i^{-1}$.
Lemma \ref{hm:aux.lem1} concludes that
\begin{equation}\label{hm:s11-8}
    \Gam_{11} \cip Q_{11,0}.
\end{equation}

Under the regularity conditions, it is easy to see that
\begin{equation}\label{hm:s11-9}
    \sup_{\eta}\left|\frac{1}{nN}\p_{\eta}^3\mbbh_{11}(\eta)\right| = O_p(1).
\end{equation}

Now, we can complete the proof of the asymptotic normality of $\wh{\eta}$ in the standard manner.
Under the consistency, we may and do suppose that $\p_{\eta}\mbbh_{11}(\wh{\eta})=0$, so that the Taylor expansion combined with \eqref{hm:s11-7}, \eqref{hm:s11-8}, and \eqref{hm:s11-9} gives
\begin{align}\label{hm:se_eta.hat}
    \wh{u}_{\eta} := \sqrt{nN}(\wh{\eta} - \eta_{0}) 
    = Q_{11,0}^{-1}\D_{11} + o_p(1) \cil N_{p_{\eta}}\left(0,Q_{11,0}^{-1}\right).
\end{align}
Based on what we have seen so far, it is easy to show that $\wh{Q}_{11}$ defined in \eqref{hm:hat-Q11} is a consistent estimator of $Q_{11,0}$.

\subsubsection{Random-effect distribution}
\label{hm:sec_diff.random.param}

This section aims to prove the asymptotic normality of $\wh{\theta}_\tau$.
Thanks to high-frequency data, we can consistently estimate each individual's random effect $\tau_i$ by $\taues_i:=\taues_i(\wh{\eta})$; recall the definition \eqref{hm:tau-hat}. A large value of $N$ enables us to make an inference on $\mcl(\tau_1)$.

Recall the definition $\mbbh_{12}(\vt_1)$ in \eqref{hm:def_QLF1-2}. 
Let $\mbbh_{12}(\theta_\tau):=\mbbh_{12}(\wh{\eta},\theta_\tau)$ with a slight abuse of notation, and write
\begin{equation}\label{hm:def_tau-hat}
    \wh{\tau}_i := \wh{\tau}_i(\wh{\eta}).
\end{equation}
By Assumption \ref{hm:A_rand.eff} and \eqref{hm:tau-moments}, we have
\begin{equation}
    \sup_{\theta_\tau}\left|\frac{1}{N}\p_{\theta_\tau}^3\mbbh_{12}(\theta_\tau)\right| = O_p(1).
\end{equation}
Introduce the log-likelihood function
\begin{equation}
    \ell_{12}(\theta_\tau) := \sumi \log f(\tau_i;\theta_\tau),
\end{equation}
which is, of course, unusable since $\tau_i$ are unobserved.
Suppose for a moment that
\begin{equation}\label{hm:1-2_goal}
    \del_{12} = \del_{12,N} := \max_{k=0,1,2} \sup_{\theta_\tau}\left| \frac{1}{\sqrt{N}}\left(
    \p_{\theta_\tau}^k \mbbh_{12}(\theta_\tau) - \p_{\theta_\tau}^k \ell_{12}(\theta_\tau)
    \right)\right| \cip 0.
\end{equation}
Then, we can proceed as if we observe $\tau_1,\dots,\tau_N$, so that we can apply the standard asymptotic theory for an i.i.d. sample under Assumption \ref{hm:A_rand.eff} to conclude the asymptotic normality:
\begin{align}\label{hm:AN_theta.tau}
    \wh{u}_{\theta_\tau}:=\sqrt{N}(\wh{\theta}_\tau - \theta_{\tau,0}) = \mci_{12}^{-1} \D_{12} + o_p(1)
    \cil N_{p_{\tau}}\left(0,\mci_{12}^{-1}\right),
\end{align}
where $\mci_{12}$ is given by \eqref{hm:def_I.12} and
\begin{align}
    \D_{12}=\D_{12}(\theta_{\tau,0}) &:= \frac{1}{\sqrt{N}} \p_{\theta_\tau} \mbbh_{12}.
\end{align}
\hmrev{To be more specific, having proved \eqref{hm:1-2_goal}, we can proceed as follows (see, for example, \cite[Chapter 5]{vdV98} for a systematic account).
\begin{itemize}
    \item Under Assumption \ref{hm:A_rand.eff}, we have the uniform law of large numbers
    \begin{equation}
        \sup_{\theta_\tau}\left| \frac{1}{N}\ell_{12}(\theta_\tau) - \E[\log f(\tau_i;\theta_\tau)] \right| \cip 0.
    \end{equation}
    Hence, \eqref{hm:1-2_goal} gives
    \begin{equation}
        \sup_{\theta_\tau}\left| \frac{1}{N}
    \mbbh_{12}(\theta_\tau) - \E[\log f(\tau_i;\theta_\tau)] \right| \cip 0.
    \end{equation}
    The last convergence combined with the identifiability condition in Assumption \ref{hm:A_rand.eff} concludes the consistency $\wh{\theta}_\tau \cip \theta_{\tau,0}$.
    \item To establish \eqref{hm:AN_theta.tau}, the standard Taylor-expansion-based argument together with the following convergences derived from \eqref{hm:1-2_goal}:
    \begin{align}
    \left| \frac{1}{\sqrt{N}}\left(
    \p_{\theta_\tau} \mbbh_{12}(\theta_{\tau,0}) - \p_{\theta_\tau} \ell_{12}(\theta_{\tau,0})
    \right)\right| &\cip 0,
    \nn\\
    \left| \frac{1}{N}\left(
    \p_{\theta_\tau}^2 \mbbh_{12}(\theta_{\tau,0}) - \p_{\theta_\tau}^2 \ell_{12}(\theta_{\tau,0}) 
    \right)\right| &\cip 0,
    \end{align}
    giving rise to the central limit theorem and the law of large numbers:
    \begin{align}
        \frac{1}{\sqrt{N}}\p_{\theta_\tau} \mbbh_{12}(\theta_{\tau,0})
        &\cil N_{p_\vp}\left(0,\mci_{12}\right), \nn\\
        -\frac{1}{N}\p_{\theta_\tau}^2 \mbbh_{12}(\theta_{\tau,0})
        &\cip \mci_{12}.
    \end{align}
\end{itemize}
}

Let us show \eqref{hm:1-2_goal}.
By a direct expansion and what we have seen in Section \ref{hm:sec_diff.fixed.param},
\begin{align}
    \wh{\tau}_i
    &= \frac1n \sum_j \left(S^{-1}_{i,j-1} 
    + \frac{1}{\sqrt{nN}}\int_0^1 \p_{\eta}(S^{-1})_{i,j-1}\left(\eta_{0})+s(\wh{\eta}-\eta_{0})\right)ds\,[\wh{u}_{\eta}]
    \right)\,\ny_{ij}^2 \nn\\
    &= \frac1n \sum_j \left(S^{-1}_{i,j-1} 
    + \frac{1}{\sqrt{nN}} \overline{\mfb}_{ij}[\wh{u}_{\eta}]\right)\,\ny_{ij}^2 
    \nn\\
    &= \frac1n \sum_j (\tau_i + \sqrt{h}\,\overline{\mfm}_{ij}) 
    + \frac{1}{\sqrt{nN}} \frac1n \sum_j \overline{\mfb}_{ij} [\wh{u}_{\eta}]\,\ny_{ij}^2 
    \nn\\
    &= \tau_i + \sqrt{h}\,\overline{\mfm}_{i} + \frac{1}{\sqrt{nN}} \frac1n \sum_j \overline{\mfb}_{ij} [\wh{u}_{\eta}]\,\ny_{ij}^2.
    \label{hm:tau.hat-s.e.}
\end{align}
Hence,
\footnote{
Recalling \eqref{hm:max.hc.mi}, we also have for any $\mathsf{c}>0$, $\max_{i\le N}|h^{\mathsf{c}-1/2}(\wh{\tau}_i - \tau_i)| \lesssim (1+N^{-1/2}|\wh{\eta}|) \max_{i\le N}|h^{\mathsf{c}}\,\overline{\mfm}_{i}|=O_p(1)o_p(1)=o_p(1)$.
}
\begin{equation}\label{hm:1-2-1}
    \left|\sqrt{n}(\wh{\tau}_i - \tau_i)\right|
    \lesssim |\overline{\mfm}_i| + \frac{|\wh{u}_{\eta}|}{\sqrt{N}} \frac1n \sum_j \ny_{ij}^2
    \lesssim |\overline{\mfm}_i| \left(1 + \frac{|\wh{u}_{\eta}|}{\sqrt{N}} \right).
\end{equation}

We introduce the events (in $\mcf$): for $K>0$,
\begin{align}
    A_{12,1} &:= \{ |\wh{u}_{\eta}| \le K\}, \nn\\
    A_{12,2} &:= \left\{ \max_{1\le i\le N}\left|\frac{\wh{\tau}_i - \tau_i}{\tau_i}\right| \le \frac12\right\}.
\end{align}
Fix any $\ep,\ep'>0$ and observe that
\begin{align}\label{hm:1-2-2}
    \pr[\del_{12} > \ep] 
    &\le \pr[A_{12,1}^c] + \pr[A_{12,1} \cap A_{12,2}^c]
    \nn\\
    &{}\qquad + \pr\left[A_{12,1} \cap A_{12,2} \cap \{\del_{12} > \ep\}\right].
\end{align}
Since $\wh{u}_{\eta}=O_p(1)$, we can take $K>0$ large enough to ensure that $\sup_{n,N}\pr[A_{12,1}^c]<\ep'$.
With this $K$, by \eqref{hm:1-2-1} and the Markov inequality we have for any $L\ge 2$,
\begin{align}
    \pr[A_{12,1} \cap A_{12,2}^c] &\le \pr\left[
    \max_{i\le N} \frac{|\overline{\mfm}_i|}{\tau_i} \ge C\sqrt{n}
    \right] \nn\\
    &\le \sumi \pr\left[\frac{|\overline{\mfm}_i|}{\tau_i} \ge C\sqrt{n}\right]
    \lesssim \frac{N}{n^{L/2}} \to 0.
\end{align}
By Assumption \ref{hm:A_rand.eff}, it is easy to see that
\begin{align}
    \pr\left[A_{12,1} \cap A_{12,2} \cap \{\del_{12} > \ep\}\right]
    \le 
    \pr\left[\sqrt{\frac{N}{n}}\,\frac1N \sumi |\overline{\mfm}_i| > \ep\right]
    \lesssim \sqrt{\frac{N}{n}} \to 0.
\end{align}
Summarizing these estimates gives $\limsup_{n,N} \pr[\del_{12} > \ep] <\ep'$, concluding \eqref{hm:1-2_goal} and thus followed by \eqref{hm:AN_theta.tau}.

\hmrev{
Finally, it is straightforward to verify that $\wh{\mci}_{12}$ given in \eqref{hm:hat-I12} is a consistent estimator of $\mci_{12}$: under Assumption \ref{hm:A_rand.eff}, by \eqref{hm:AN_theta.tau} and \eqref{hm:tau.hat-s.e.},
\begin{align}
\wh{\mci}_{12} 
    &= \frac{1}{N} \sumi \left(\p_{\theta_\tau}\log f(\tau_i;\theta_{\tau,0})\right)^{\otimes 2} \nn\\
    &{}\qquad 
    +\left(\frac{1}{N} \sumi \left(\p_{\theta_\tau}\log f(\wh{\tau}_i;\wh{\theta}_{\tau})\right)^{\otimes 2}  - \frac{1}{N} \sumi \left(\p_{\theta_\tau}\log f(\tau_i;\theta_{\tau,0})\right)^{\otimes 2} \right)
    \nn\\
    &= \frac{1}{N} \sumi \left(\p_{\theta_\tau}\log f(\tau_i;\theta_{\tau,0})\right)^{\otimes 2} + o_p(1) \cip \mci_{12}.    
\end{align}
}

\subsubsection{Joint asymptotic normality}

Having obtained the asymptotic linear representations \eqref{hm:s11-9} and \eqref{hm:AN_theta.tau}, it remains to show the asymptotic orthogonality, hence the asymptotic independence of $\wh{\eta}$ and $\wh{\theta}_\tau$. We only need to observe that the covariance tends to $0$.

By what we have seen in Sections \ref{hm:sec_diff.fixed.param} and \ref{hm:sec_diff.random.param}, in particular the stochastic expansions \eqref{hm:se_eta.hat} and \eqref{hm:AN_theta.tau}, the computation of the covariance of the $\D_{11}$ and $\D_{12}$ boils down to that of
\begin{align}
    \cov[\zeta_i,\zeta_i'] 
    &= \frac12\,
    \cov\left[ M_{11,1} - \overline{g}_1 M_{11,0},\, \p_{\theta_\tau}\log f(\tau_1;\theta_{\tau,0})\right].
\end{align}
Indeed, it can be seen that $\lim_{n,N}\cov[\zeta_i,\zeta_i']=0$ ($p_{\eta}\times p_{\tau}$-zero matrix) in a similar manner to the previous $\mcf$-stable convergence argument of $(M_{11,0},M_{11,1})$ combined with the uniform integrability that we have used to obtain the limit of the expectation in \eqref{hm:mart.rep-5}.
This completes the proof of Theorem \ref{hm:thm_diff.param}.

\subsection{Proof of Theorem \ref{hm:thm_drif.param}}

To study the asymptotic behavior of second-stage marginal quasi-likelihood $\mbbh_{2}(\vt_2)$ defined by \eqref{hm:joint.qlm-2}, we will proceed as follows.
\begin{enumerate}
    \item First, we look at the stochastic expansions of $\wh{M}_i$ and $\wh{v}_i$:
    \begin{align}
        \wh{M}_i = M_{i,T} + \mathsf{R}^M_i, 
        \qquad 
        \wh{v}_i = v_{i,T} + \mathsf{R}^v_i,
    \end{align}
    for specific leading terms $M_{i,T}$ and $v_{i,T}$ (recall their definitions in \eqref{hm:M&v_def}) and some remainder (negligible) terms $\mathsf{R}^M_i$ and $\mathsf{R}^v_i$.
    
    \item Second, we estimate the gap between $\mbbh_2(\vt_2)$ and
    \begin{align}
        \mbbh^\star_{2}(\vt_2) 
        &:= \sum_i \log \phi_{p_{\vp}}\big(M_{i,T}^{-1} v_{i,T} ;\, \mu,\, M_{i,T}^{-1}+\Sig\big)
        \nn\\
        &= \sum_i \log \phi_{p_{\vp}}\big(M_{i,T}^{-1} v_{i,T} ;\, (\vp_f,0),\, M_{i,T}^{-1}+\diag(O,\Sig_r)\big)
        \label{hm:joint.qlm-2.star}
    \end{align}
    together with their derivatives, concluding that they are negligible uniformly in $\vt_2$.
    
    \item Then, we conclude the asymptotic normality of $\wh{\vt}_2$ by investigating $\mbbh^\star_{2}(\vt_2)$ instead of $\mbbh_{2}(\vt_2)$.
\end{enumerate}

\subsubsection{Stochastic expansions of $\wh{M}_i$ and $\wh{v}_i$}
\label{hm:sec_se-hM.hv}

Recalling the chain of equalities \eqref{hm:tau.hat-s.e.}, we obtain
\begin{align}
\wh{M}_i 
&= \wh{\tau}_i\,h \sum_j \wh{S}_{i,j-1}^{-1} a_{i,j-1}^{\otimes 2} 
\nn\\
&= \left(\tau_i + \sqrt{h}\,\overline{\mfm}_{i} + \frac{1}{\sqrt{nN}} \overline{\mfm}_{i} \,[\wh{u}_{\eta}]\right)\, h \sum_j \left(S^{-1}_{i,j-1} + \frac{1}{\sqrt{nN}} \overline{\mfb}_{ij}[\wh{u}_{\eta}]\right) a_{i,j-1}^{\otimes 2} 
\nn\\
&= \tau_i h \sum_j S_{i,j-1}^{-1} a_{i,j-1}^{\otimes 2} + \check{\mathsf{R}}^M_i,
\end{align}
where
\begin{equation}\label{hm:drif_p1.1}
    |\check{\mathsf{R}}^M_i| \le \sqrt{h}\, |\overline{\mfm}_{i}| \left(1+\frac{1}{\sqrt{N}}|\wh{u}_\eta|\right).
\end{equation}
Then, applying \eqref{hm:aux.lem1-0} in Lemma \ref{hm:aux.lem1} gives
\begin{equation}
    \wh{M}_i = \tau_i \int_0^T S_i(t)^{-1} a_i(t)^{\otimes 2} dt + \mathsf{R}^M_i = M_{i,T} + \mathsf{R}^M_i
\end{equation}
with the upper bound of $|\mathsf{R}^M_i|$ having the same form as in \eqref{hm:drif_p1.1}.

\smallskip

Turning to $\wh{v}_i$, using the second-order expansion
\begin{align}
    \wh{S}^{-1}_{i,j-1}
    &= S^{-1}_{i,j-1} + \frac{1}{\sqrt{nN}} (\p_\eta (S^{-1}))_{i,j-1}[\wh{u}_{\eta}] + \frac{1}{nN}\overline{\mfb}_{ij}[\wh{u}_{\eta}^{\otimes 2}],
\end{align}
we get
\begin{align}
\wh{v}_i 
&= \sum_j \wh{S}_{i,j-1}^{-1} a_{i,j-1} \D_j Y_i \nn\\
&= \sum_j S_{i,j-1}^{-1} a_{i,j-1} \D_j Y_i + \frac{1}{\sqrt{nN}} \underbrace{\left(\sum_j (\p_\eta (S^{-1}))_{i,j-1} a_{i,j-1} \D_j Y_i\right)}_{=: v_{1,i}}\, [\wh{u}_{\eta}]
\nn\\
&{}\qquad + \underbrace{\left(\frac{1}{nN} \sum_j \overline{\mfb}_{ij} a_{i,j-1} \D_j Y_i\right)}_{=: v_{2,i}}\, [\wh{u}_{\eta}^{\otimes 2}].
\end{align}
In a similar manner to \cite[Lemma 2]{delattre2018parametric}, we can show that
\begin{equation}
    \frac{1}{\sqrt{h}} \left( \sum_j \mathsf{f}_{i,j-1}\D_j Y_i - \int_0^T \mathsf{f}(Y_i(t)) dY_i(t) \right) = \overline{\mfm}_{i}
\end{equation}
for any measurable $\mathsf{f}(y)$ satisfying $\sup_i \sup_{t\le T}\E[|\mathsf{f}(Y_i(t))|^K]<\infty$ for every $K>0$. This implies that $v_{1,i}=\overline{\mfm}_{i}$. It is easy to see that
\begin{equation}
    v_{2,i} = \frac{\sqrt{h}}{nN} \sum_j \overline{\mfb}_{ij} a_{i,j-1} \ny_{ij} = \frac{\sqrt{h}}{N} \, \overline{\mfm}_{i}.
\end{equation}
It follows that
\begin{align}
    \wh{v}_i &= \sum_j S_{i,j-1}^{-1} a_{i,j-1} \D_j Y_i 
    + \frac{1}{\sqrt{nN}} \overline{\mfm}_{i}\, [\wh{u}_{\eta}]
    + \frac{\sqrt{h}}{N} \, \overline{\mfm}_{i} [\wh{u}_{\eta}^{\otimes 2}]
    \nn\\
    &= \int_0^T S_i(t)^{-1} a_i(t) dY_i(t) + \mathsf{R}^v_i = v_{i,T} + \mathsf{R}^v_i,
\end{align}
where
\begin{equation}
    |\mathsf{R}^v_i| \le \sqrt{h}\, |\overline{\mfm}_{i}| \left(1+\frac{1}{\sqrt{nN}}|\wh{u}_\eta|+\frac{\sqrt{h}}{N}|\wh{u}_\eta|^2\right).
\end{equation}

In sum, we obtained
\begin{align}
    & |\wh{M}_i - M_{i}| \le \sqrt{h}\, |\overline{\mfm}_{i}| \left(1+\frac{1}{\sqrt{N}}|\wh{u}_\eta|\right), 
    \label{hm:added_M.hat-M_gap}\\
    & |\wh{v}_i - v_{i}| \le \sqrt{h}\, |\overline{\mfm}_{i}| \left(1+\frac{1}{\sqrt{nN}}|\wh{u}_\eta|+\frac{\sqrt{h}}{N}|\wh{u}_\eta|^2\right).
\end{align}

\subsubsection{Estimating the gap}

From the expressions \eqref{hm:joint.qlm-2} and \eqref{hm:joint.qlm-2.star}, some calculations give the estimate
\begin{align}
    & \sup_{\vt_2}\left|\frac{1}{\sqrt{N}}\left( \mbbh_{2}(\vt_2) - \mbbh^\star_{2}(\vt_{2}) \right)\right| \nn\\
    &\lesssim \frac1N \sum_i \sqrt{N}\left|\log|\wh{M}_i| - \log|M_i| \right|
    \nn\\
    &{}\qquad + \sup_{\Sig_r}\frac1N \sum_i \sqrt{N}\left|
    \log|\Sig \wh{M}_i + I_{p_\vp}| - \log|\Sig M_i + I_{p_\vp}| \right|
     \nn\\
    &{}\qquad + \sup_{\vt_2}\frac1N \sum_i \sqrt{N}\Big|
    \mathrm{Trace}\left(
    (\Sig+\wh{M}_i^{-1})^{-1}(\wh{M}_i^{-1}\wh{v}_i - \mu)^{\otimes 2}
    \right)
    \nn\\
    &{}\qquad \qquad 
    \mathrm{Trace}\left(
    (\Sig+M_i^{-1})^{-1}(M_i^{-1} v_i - \mu)^{\otimes 2}
    \right)\Big|
    \nn\\
    &=: \mcg_{1,1} + \mcg_{1,2} + \mcg_2.
\end{align}
To deduce that all three terms on the rightmost side are negligible, we assume without loss of generality that the variables are scalar.

Observe that
\begin{align}
    \mcg_{1,1} &= \frac1N \sum_i \sqrt{N}\left|
    \log\left(1+\frac{\wh{M}_i - M_i}{M_i}\right)\right|.
\end{align}
By \eqref{hm:sampling.design}, we have $N/n^{\mathsf{a}''} \lesssim 1$ ($\mathsf{a}''\in(0,1)$). 
By \eqref{hm:max.hc.mi} and \eqref{hm:added_M.hat-M_gap},
\begin{equation}
    \max_{i\le N}|\wh{M}_i-M_i| =o_p(h^{1/2-\mathsf{c}})
\end{equation}
for any (small) $\mathsf{c}>0$, hence in particular taking $\mathsf{c}=(1-\mathsf{a}'')/2$ gives
\begin{equation}\label{hm:drif-p1}
    \sqrt{N} \max_{i\le N}|\wh{M}_i-M_i| = o_p\big(\sqrt{N} h^{1/2-\mathsf{c}}\big) 
    = o_p\big(N/n^{\mathsf{a}''}\big) = o_p(1).
\end{equation}
This and Assumption \ref{hm:A_drift} imply that $\pr[B_n]\to 0$ where
\begin{equation}
    B_n := \left\{
    \max_{i\le N} \left| \frac{\wh{M}_i - M_i}{M_i} \right| > h^{\mathsf{a}''/2}
    \right\}.
\end{equation}
For any $\mathsf{Q}>0$ and sufficiently large $n$,
\begin{equation}
    \pr[\mcg_{1,1} > \mathsf{Q}] \le \pr[B_n] + \pr\left[B_n^c \cap\{\mcg_{1,1} > \mathsf{Q}\}\right] \to 0,
\end{equation}
concluding that $\mcg_{1,1}=o_p(1)$. 
Similarly, we can prove that $\mcg_{1,2}=o_p(1)$ (much easier since $\Sig \wh{M}_i + I_{p_\vp} \ge I_{p_\vp} > O$). 
As for $\mcg_{2}$, under Assumption \ref{hm:A_drift}, some manipulations and estimates together with \eqref{hm:drif-p1} give
\begin{align}\label{hm:drif-p2}
    \mcg_{2} &\lesssim o_p(1) \,
    \frac1N \sum_i |\overline{\mfm}_i| |M_i^{-1}| |\wh{M}_i^{-1}| (1+|M_i^{-1}|)
    \nn\\
    & \le o_p(1)\, \frac1N \sum_i |\overline{\mfm}_i| = o_p(1).
\end{align}
Thus, it follows that
\begin{equation}\label{hm:gap.estimate_0}
    \sup_{\vt_2}\left|\frac{1}{\sqrt{N}}\left( \mbbh_{2}(\vt_2) - \mbbh^\star_{2}(\vt_{2}) \right)\right| = o_p(1).
\end{equation}
Analogously, we can also deduce that
\begin{equation}\label{hm:gap.estimate_123}
    \sup_{\vt_2}\left|\frac{1}{\sqrt{N}}\left( \p_{\vt_2}^k \mbbh_{2}(\vt_2) - \p_{\vt_2}^k \mbbh^\star_{2}(\vt_{2}) \right)\right| = o_p(1)
\end{equation}
for $k=1,2,3$.

\subsubsection{Asymptotic normality of $\wh{\vt}_2$}

Now we are in the position to conclude the consistency and the asymptotic normality of $\wh{\vt}_2$. 
Building on \eqref{hm:gap.estimate_0} and \eqref{hm:gap.estimate_123}, under Assumption \ref{hm:A_drift} we can follow the standard $M$-estimation machinery for the i.i.d. model associated with the random function $\mbbh^\star_{2}(\vt_{2})$ (see, for example, \cite[Chapter 5]{vdV98}). 
Finally, to derive the form of the asymptotic covariance matrix, we can follow the same line as in \cite[Section 6.2]{DelGenLar18-2}.

\subsection*{Acknowledgement}
We are grateful to the INRAE MIGALE bioinformatics platform \url{https://migale.inrae.fr} for providing computational resources.

This work was partially supported by JST CREST Grant Number JPMJCR2115, Japan, and by JSPS KAKENHI Grant Numbers 23K22410 and 24K21516 (HM).

\bibliographystyle{abbrv} 
\bibliography{ql-sdeme-ejs}

\end{document}